\documentclass[11pt]{article}
\usepackage{graphicx} 
\usepackage{xcolor}
\usepackage{subcaption}
\usepackage{amsfonts, amsmath, amssymb, bbm}
\usepackage{caption, subcaption}
\usepackage{todonotes}
\usepackage{comment}
\usepackage[a4paper, total={6in, 10in}]{geometry}
\usepackage[backend=biber]{biblatex}
\usepackage{amsthm}
\usepackage{algpseudocode}
\usepackage{tabularx}
\usepackage{siunitx}
\usepackage{booktabs}
\usepackage{algorithm}
\usepackage[most]{tcolorbox}
\usepackage[colorlinks=true, allcolors=blue]{hyperref}
\newtheorem{theorem}{Theorem}[section]
\tcolorboxenvironment{theorem}{
  colback=gray!8,
  colframe=black,
  boxrule=0.8pt,
  arc=2mm,
  left=2mm,
  right=2mm,
  top=1mm,
  bottom=1mm
}
\newtheorem{prop}{Proposition}[section]
\tcolorboxenvironment{prop}{
  colback=gray!8,
  colframe=black,
  boxrule=0.8pt,
  arc=2mm,
  left=2mm,
  right=2mm,
  top=1mm,
  bottom=1mm
}

\newtheorem{lemma}[theorem]{Lemma}
\tcolorboxenvironment{lemma}{
  colback=gray!1,
  colframe=black,
  boxrule=0.8pt,
  arc=2mm,
  left=2mm,
  right=2mm,
  top=1mm,
  bottom=1mm
}

\title{A Discontinuous Galerkin discretization for the sea ice dynamics}
\author{Emma Lagracie, Thomas Richter}
\author{Emma Lagracie$^{1,} \thanks{emma.lagracie@ovgu.de}\text{ }$ and Thomas Richter$^{1, }\thanks{thomas.richter@ovgu.de}\text{ }$ \\
\ \\ 
$^1$ Otto von Guericke Universität, Magdeburg, Germany\\}

\newcommand{\bu}{\mathbf{u}}
\newcommand{\bv}{\mathbf{v}}
\newcommand{\bt}{\mathbf{t}}
\newcommand{\bI}{\mathbf{I}}

\newcommand{\sig}{\boldsymbol{\sigma}}
\newcommand{\btau}{\boldsymbol{\tau}}
\newcommand{\beps}{\boldsymbol{\varepsilon}}

\newcommand{\ljp}{ [\![ }
\newcommand{\rjp}{ ]\!] }

\newcommand{\Tr}{\text{Tr}}

\begin{document}
\maketitle

\begin{abstract}
    Sea ice dynamics plays a crucial role in the Earth's climate system, making it an important component of weather and climate prediction models. Its numerical simulation remains challenging, however, as it exhibits complex mechanical behaviors due to a nonlinear ice rheology and interactions with external physical forcings. In particular, sea ice  presents linear kinematic features (LKFs), i.e., narrow bands of intense deformation associated with processes such as lead opening or pressure-ridge formation. Accurately representing these features is necessary as they affect thermodynamics and ocean–atmosphere exchange. Yet their number, localization, and structure are highly sensitive to spatial resolution and to the chosen discretization of the sea ice velocity.

    In this work, we investigate the spatial discretization of Hibler’s viscous plastic sea-ice model using a fully discontinuous Galerkin (DG) representation of all variables, including the velocity field. The ability of DG elements to represent discontinuities while having high order local polynomial approximation makes them well suited for resolving the complex ice deformation features, and insuring robustness towards mesh-induced numerical artefacts. We assess the fully DG method using an established sea ice dynamics benchmark and compare the obtained sea ice deformation with state-of-the-art discretizations.

    Additionally, we study the theoretical convergence of the modified Elastic-Viscous-Plastic (mEVP) formulation as a pseudo-time iterative solver for the viscous-plastic (VP) momentum equations. We prove the convergence of the underlying continuous pseudo-time dynamical system towards the viscous plastic formal limit, thereby providing a theoretical foundation for the use of mEVP as an iterative solver for Hibler’s VP model.

    
\end{abstract}

\section{Introduction}

\paragraph{Sea ice dynamics and the need for accurate numerical models}
Sea ice plays a crucial role in the Earth's climate system, influencing the global energy budget through its high albedo, regulating ocean-atmosphere heat and mass exchanges, and governing large-scale circulation patterns in the polar oceans. Accurate numerical simulation of sea ice is therefore an important component of global climate models and regional ocean forecasting systems. At the same time, sea ice exhibits remarkably complex mechanical behavior: it deforms as a highly heterogeneous material, developing sharp localized structures such as leads --- open channels of water cutting through the ice pack --- and linear kinematic features (LKFs), which are narrow bands of intense deformation that govern much of the large-scale ice dynamics. These features emerge from the strongly nonlinear interaction between wind and ocean forcing, internal ice stress, and the thermodynamic evolution of the ice cover. Resolving them faithfully is one of the central challenges in sea ice numerics, requiring both high spatial resolution and discretizations that are sensitive to such localized physical phenomena. In this paper, we present a fully discontinuous Galerkin (DG) discretization of the viscous-plastic sea ice model, designed to address precisely these requirements.

\paragraph{The viscous-plastic sea ice model}
The most widely used model for sea ice rheology is the viscous-plastic (VP) model introduced by Hibler~\cite{Hibler1979,Hibler1980}, which describes sea ice as a two-dimensional continuum with a nonlinear viscous-plastic constitutive law. In this framework, internal ice stresses are determined by a pressure-dependent yield curve --- typically an elliptical yield curve in principal stress space --- and the ice deforms plastically when the stress state reaches the yield curve, while exhibiting a highly viscous (near-rigid) response in the interior. The governing equations couple the ice momentum balance, including wind and ocean drag and the Coriolis force, to transport equations for ice thickness and concentration that advect these scalar quantities with the ice velocity. Despite the apparent simplicity of its continuum formulation, the VP model poses severe numerical challenges. The constitutive law involves strongly nonlinear, deformation-rate-dependent viscosities that can vary by many orders of magnitude across the domain, giving rise to nearly singular operators and making the solution of the discrete system extremely difficult. Alternative rheological models have been proposed in recent years — most notably the Maxwell Elasto-Brittle (MEB) and Brittle Bingham Maxwell (BBM) models developed by Dansereau, Rampal, Olason and collaborators~\cite{Dansereau2016,Rampal2016,Olason2022}, which aim to better represent the fracture-like behavior of sea ice — but the VP model remains the dominant framework in operational and climate modeling due to its established physical basis, widespread implementation, and extensive validation record.

\paragraph{Solving the VP model: implicit and iterative solvers}
The difficulty of solving the VP momentum equation has driven sustained efforts in numerical solver development. The naive approach of fixed-point (Picard) iteration~\cite{ZhangHibler1997} converges slowly or not at all in the strongly nonlinear regime, and has long been recognized as inadequate for production use~\cite{LemieuxTremblay2009}. Significant progress has been achieved with fully implicit nonlinear solvers: Lemieux and Losch pioneered the use of Jacobian-free Newton–Krylov (JFNK) methods~\cite{Lemieux2012,Losch2014} for VP sea ice, demonstrating dramatically improved convergence properties over Picard iteration. Subsequently, modified Newton approach was developed~\cite{Mehlmann2017} that exploits the specific structure of the VP nonlinearity, and more recently Shih et al. proposed a primal-dual active set method that handles the plastic–viscous transition in a rigorous optimization framework~\cite{Shih2023}. These implicit solvers are capable of delivering accurate, well-converged solutions, but their practical deployment remains challenging: the underlying nonlinear systems are difficult to precondition effectively, the solvers require careful implementation and tuning, and maintaining robust, scalable code across different model configurations imposes a significant engineering burden. As a result, the dominant approach in operational practice remains a class of explicit-iterative methods. The most widely adopted is the elastic-viscous-plastic (EVP) scheme of Hunke and Dukowicz~\cite{HunkeDukowicz1997}, which introduces an artificial elastic term into the VP constitutive law, allowing the stress components to be treated as independent prognostic variables and the momentum equation to be integrated explicitly in pseudo-time. The modified EVP (mEVP) formulation of Bouillon and collaborators~\cite{Bouillon2013} later resolved the dependence on an explicit elastic wave speed by introducing a secondary pseudo-time scale, yielding a scheme that formally converges to VP solutions when a sufficient number of sub-cycling iterations are performed~\cite{Kimmritz2016,Kimmritz2017}. These EVP-type methods are computationally simple, easy to implement, and naturally parallelizable, but they require a large number of sub-iterations per physical time step to achieve adequate numerical convergence, and the required number grows with mesh refinement, making them increasingly expensive as resolution increases.

\paragraph{Theoretical convergence of the dynamical mEVP system}
The mEVP formulation can be interpreted as a time-discrete approximation of an EVP-like dynamical system, evolving in artificial pseudo-time over each fixed physical time step. Previous studies have investigated the numerical convergence and stability of the discrete mEVP system, establishing stability conditions for several spatial grid arrangements and demonstrating numerically that, provided a sufficiently large number of pseudo-time iterations is performed, the discrete mEVP solution converges toward its formal VP limit \cite{kimmritz2015convergence,Kimmritz2016}. However, to the best of our knowledge, the mathematical convergence of the underlying continuous pseudo-time dynamical system toward a stable equilibrium corresponding to the VP limit, independently of both the spatial and pseudo-time discretizations, has not previously been established.

\paragraph{Discretization and its effect on sea ice solutions}
Beyond the question of which solver is used, the spatial discretization has a profound influence on the quality of sea ice simulations, and in particular on the representation of leads and LKFs~\cite{MehlmannLKF2021}. These features are intrinsically sharp: leads are free surfaces across which the ice velocity may be nearly discontinuous, and LKFs are localized bands of plastic deformation with widths that may approach the grid scale. Systematic studies have demonstrated that standard finite difference and finite element discretizations on structured grids require very high resolution — on the order of a few kilometers or less — to begin resolving these structures, and that coarser models systematically underestimate deformation rates and smooth out the spatial intermittency of sea ice dynamics. The dependence on resolution and discretization choice has been quantified in dedicated LKF intercomparison studies by Mehlmann et al.~\cite{MehlmannLKF2021}, confirming that the statistical properties of LKFs, including their length distributions and occurrence frequencies, are strongly sensitive to grid spacing. Importantly, it has also been shown that the type of discretization matters beyond mere resolution: the use of non-conforming elements, such as Crouzeix–Raviart (CR) elements, which allow inter-element velocity discontinuities, leads to sharper and more physically consistent LKF structures compared to standard continuous Galerkin discretizations at the same resolution~\cite{Mehlmann2021}. The neXtSIM-DG project~\cite{nextsimdg} further demonstrated that increasing the polynomial order of the discretization, even at fixed mesh resolution, improves the representation of sea ice deformation features~\cite{richter2023dynamical}, suggesting that high-order methods offer an efficient path to improved solution quality without the full cost of mesh refinement.

\paragraph{Discontinuous Galerkin methods for sea ice}

The evidence from these studies points clearly toward discontinuous Galerkin (DG) methods as a natural framework for sea ice numerics~\cite{di2011mathematical}. DG methods combine high polynomial order with a completely local, element-wise approximation that allows inter-element discontinuities — precisely the combination indicated by the LKF studies~\cite{MehlmannLKF2021}. The discontinuous character of DG approximations is physically motivated: sea ice deformation is genuinely discontinuous at leads, and forcing the numerical solution to be globally continuous artificially constrains the solution space. The flexibility of DG discretizations also makes them well suited to the hyperbolic transport equations for ice thickness and concentration. Dansereau~\cite{Dansereau2016,Dansereau2017,Dansereau2021} was the first to systematically investigate DG for discretizing the tracer equations of sea ice, also exploiting the role of higher order. In the neXtSIM-DG model~\cite{nextsimdg}, DG has been used to discretize the tracer transport but also the stress and strain~\cite{richter2023dynamical} while retaining continuous finite elements for the velocity. To our knowledge, no fully DG formulation of the VP sea ice model — in which all fields, including velocity, are represented in discontinuous approximation spaces — has been previously reported.

\paragraph{Outline of this paper}
In this paper, we close this gap by presenting the first fully DG discretization of Hibler's VP sea ice model, based on the local discontinuous Galerkin (LDG) framework~\cite{cockburn1998local}. The key idea is to exploit the natural first-order structure of the VP constitutive law: by introducing the strain rate tensor as an independent variable, the momentum equation can be written as a first-order system in which both the velocity and the strain rate are coupled through the constitutive relation and the divergence operator. The LDG formulation discretizes this system directly, with velocity, strain rate, and stress all living in DG spaces, and achieves the necessary coupling between elements through numerical flux terms.  

In addition, we provide a mathematical justification for the convergence of the mEVP solver, by studying the underlying pseudo-time dynamical system. Using Lyapunov functionals, we prove the convergence of the mEVP velocity and stress towards the physical time discrete solution of the VP system, as pseudo time goes to infinity.

The paper is organized as follows. Section 2 introduces the sea ice model in VP and mEVP formulation as well as the necessary notation. Section 3 is dedicated to proving the convergence of the mEVP system towards its VP formal limit. Section 4 presents the mixed form DG discretization for Hibler's sea ice model, and numerical results on a classical benchmark are presented in Section 5. Section~6 concludes with a discussion and directions for future work.


\section{The sea ice model}

We consider the standard VP sea ice model introduced by Hibler~\cite{Hibler1979}, which describes the large-scale dynamics of sea ice as a two-dimensional viscous-plastic continuum on the surface of the ocean. The model consists of a momentum equation for the ice velocity $\bu : \Omega \to \mathbb{R}^2$, coupled to transport equations for mean ice thickness $H$ and ice concentration
$A$, defined on a two-dimensional domain $\Omega \subset \mathbb{R}^2$. The momentum equation reads
\begin{equation}
  \rho_{\mathrm{ice}} H \left( \partial_t \bu
  + f \vec k \times \bu \right)
  = \nabla \cdot \sig + A(\btau_a + \btau_o) + \rho_{\mathrm{ice}} H f \vec k \times \bu_o,
  \label{eq:momentum}
\end{equation}
where $\rho_{\mathrm{ice}}$ is the ice density, $H$ the mean ice thickness, $f$ the Coriolis parameter, and $\vec{k}$ the unit vertical vector. The right-hand side includes the divergence of the internal ice stress tensor $\sig$, the atmospheric wind stress $\btau_a$, and the ocean drag $\btau_o$ are given as
\begin{equation}
\btau_a = \rho_a C_a\|\bu_a-\bu\|_2(\bu_a-\bu),\quad
\btau_o = \rho_o C_o\|\bu_o-\bu\|_2(\bu_o-\bu),
\end{equation}
where $\bu_a$ and $\bu_o$ are wind velocity and ocean surface current and $C_a$ and $C_o$ a drag coefficients. Usually $\|\bu_a\|\gg \|\bu\|$ and often, the approximation $\btau_a=\|\bu_a\|\bu_a$ is considered. We will jointly write all forces that are not related to the internal stress as
\begin{equation}\label{forces}
F(\bu) = A(\btau_a(\bu)+\btau_o(\bu)) + \rho_\text{ice} H f\vec k\times(\bu_o-\bu).
\end{equation}
The transport equations for $H$ and $A$ take the standard advective form and are not the focus of this work; we refer
to~\cite{Hibler1979,Hibler1980} for their precise statement and both Section 4 and ~\cite{richter2023dynamical} for the higher order discontinuous Galerkin formulation.

\paragraph{Rheology}

The constitutive law relating the stress tensor $\sig$ to the
ice velocity is the VP rheology. The strain rate tensor is defined as
\begin{equation}
  \beps(\bu)
  = \frac{1}{2}\left(\nabla \bu+ (\nabla \bu)^T\right),
  \label{eq:strain}
\end{equation}
and the stress is determined by
\begin{equation}
  \sig
  = 2\eta\,\beps
  + (\zeta - \eta)(\Tr(\beps))\bI
  - \frac{P}{2}\bI,
  \label{eq:stress}
\end{equation}
where $\boldsymbol{I}$ denotes the identity tensor. The bulk and shear
viscosities $\zeta$ and $\eta$ are nonlinear functions of the strain rate,
given by
\begin{equation}
  \zeta = \frac{P}{2\Delta}, \qquad \eta = \frac{\zeta}{e^2},
  \label{eq:viscosities}
\end{equation}
where $e$ is the aspect ratio of the elliptical yield curve and $\Delta$ is the effective deformation rate
\begin{equation}
  \Delta = \left[
  \Delta_\text{min}^2 + 2e^{-2}\beps':\beps'+\Tr(\beps)^2\right]^{1/2},
  \label{eq:Delta}
\end{equation}
where $\beps' = \beps-\frac{1}{2}\text{Tr}(\beps)\,I$ is trace free, and $\Delta_\text{min}>0$ --- usually $\Delta_\text{min}=2\cdot 10^{-9}$ --- can either be considered a numerical stabilization parameter or the transition to the viscous regime~\cite{Hibler1979}, limiting the viscosity. 
The ice strength $P$ depends on the ice state through the empirical relation
\begin{equation}
  P = P^* H \exp\!\left(-C(1 - A)\right),
  \label{eq:icestrength}
\end{equation}
with material constants $P^*$ and $C$. Common choices for the physical parameters are given in Table~\ref{tab:mevp_parameters} in the numerics section. In the following, the VP stress tensor will be denoted $\sig^{VP}(\beps)$ to emphasize its nonlinear dependence on the velocity strain tensor.

\paragraph{mEVP reformulation}

The strong nonlinearity of the VP constitutive law makes direct solution of the momentum equation~\eqref{eq:momentum} challenging. Following Bouillon et al.~\cite{Bouillon2013}, we consider the modified elastic-viscous-plastic (mEVP) approach, which introduces an artificial elastic relaxation in pseudo-time to enable efficient explicit sub-cycling. The stress is promoted to an independent
variable and evolved according to
\begin{equation}
  \alpha \left(\sig^{k+1} - \sig^k\right)+ \sig^{k+1}
  = \sig^{VP}\!\left(\beps(\bu^k)\right),
  \label{eq:mevp_stress}
\end{equation}
where $k$ denotes the sub-iteration index and $\alpha > 0$ distinguishes the physical from the mEVP time scale. The momentum equation is updated simultaneously:
\begin{equation}
  \beta\,\rho_\text{ice}H\frac{\bu^{k+1} - \bu^k}{\Delta t}
  + \rho_{\mathrm{ice}} H\,\frac{\bu^{k+1} - \bu^n}{\Delta t}
  = \nabla \cdot \sig^{k+1} + F(\bu^k),
  \label{eq:mevp_momentum}
\end{equation}
where $\bu^n$ denotes the velocity at the previous physical time
step $\Delta t$ and $\beta > 0$ is a second stabilization parameter. Upon convergence as $k \to \infty$, both iterates formally recover the VP solution at the current physical time step. The parameters $\alpha$ and $\beta$ govern the convergence rate and must be chosen sufficiently large to ensure stability; we refer to~\cite{Bouillon2013,Kimmritz2016} for a detailed discussion. A key advantage of the mEVP formulation is that, unlike the original EVP scheme~\cite{HunkeDukowicz1997}, it does not introduce a physical elastic wave speed and the sub-iteration is a pure pseudo-time relaxation numerically converging to the VP solution \cite{kimmritz2015convergence}. A mathematical proof for the convergence of the mEVP system is given in the following section.


\section{Convergence of the dynamical mEVP system}

We consider the sea-ice velocity $\bu$ in the space of $L^2$-functions with $L^2$-gradient over the bounded domain $\Omega$, and subject to null Dirichlet boundary condition $(H^1_0(\Omega))^2$. The stress tensor $\sig$ is taken in the $(H^{\operatorname{div}}(\Omega))^{2\times2}$ space, with $H^{\operatorname{div}}(\Omega)):=\left \{ \sig \in L^2(\Omega)\, | \, \operatorname{div}(\sig) \in L^2(\Omega) \right \}$. In the following, we always assume that the $(H^1_0(\Omega))^2 \times (H^{\operatorname{div}}(\Omega))^{2\times2}$ solution of the viscous plastic limit system exist. We refer to \cite{liu2022well, brandt2022rigorous} for the analysis of the viscous plastic model, and to \cite{boutros2026global} for the analysis of the elastic-viscous-plastic model.

Taking $\alpha = \beta = \frac{1}{\tau}$ in \eqref{eq:mevp_stress} and \eqref{eq:mevp_momentum}, where $\tau$ is a pseudo-time, the mEVP equations  
clearly appear to be the pseudo-time explicit Euler discretization of the continuous dynamical system 
\begin{equation*}
\begin{aligned}
\rho_\text{ice} H \partial_\tau \bu+ \rho_\text{ice} H\bu-\Delta t \operatorname{div}(\sig)&=\Delta t\, F(\bu) + \rho_\text{ice} H\bu^{n-1}, \\
\partial_\tau \sig+\sig-\sig^{\mathrm{VP}}(\beps)&=0.
\end{aligned}
\end{equation*}Its formal limit, obtained for $\partial_\tau \bu = \partial_\tau \sig = 0$, is given by
\begin{equation*}
\begin{aligned}
\rho_\text{ice} H \bu^* -\Delta t \operatorname{div}(\sig^*)&=\Delta t\, F(\bu^*) + \rho_\text{ice} H\bu^{n-1}, \\
\sig^*-\sig^{\mathrm{VP}}(\beps^*)&=0,
\end{aligned}
\end{equation*} and corresponds to the VP momentum equation \eqref{eq:momentum} discretized in physical time $t$ with an implicit Euler scheme. 
The following convergence can be established:
\begin{theorem}[mEVP convergence]
\label{thm:VP cv}For external forcings of the form 
\[
F(\bu) = A \|\bu_o-\bu^n\|_2(\bu_o-\bu^n) + \tilde{f},
\]
the solution $(\bu(\tau), \sig(\tau)) \in (H^1_0(\Omega))^2 \times (H^{\operatorname{div}}(\Omega))^{2\times2}$ of the mEVP continuous pseudo time model
\begin{equation}
\begin{aligned}
&\rho_\text{ice} H \partial_\tau \bu+ \rho_\text{ice} H\bu-\Delta t \operatorname{div}(\sig)=\Delta t\, F(\bu) + \rho_\text{ice} H\bu^{n-1}, \\
&\partial_\tau \sig+\sig-\sig^{\mathrm{VP}}(\beps)=0,
\end{aligned}
\label{continuous-mEVP complete}
\end{equation}
converges for $\tau \rightarrow \infty$ in $L^2(\Omega) \times H^{\operatorname{div}}(\Omega)$ sense with exponential rate towards the solution $(\bu^*, \sig^*) \in (H^1_0(\Omega))^2 \times (H^{\operatorname{div}}(\Omega))^{2\times2}$ of the semi-discretized VP model with  time step $\Delta t$, taken at physical time $t^n$:
\begin{equation}
\begin{aligned}
&\rho_\text{ice} H \bu^* -\Delta t \operatorname{div}(\sig^*)=\Delta t\, F(\bu^*) + \rho_\text{ice} H\bu^{n-1}, \\
&\sig^*-\sig^{\mathrm{VP}}(\beps^*)=0.
\end{aligned}
\label{semi-discrete VP limit}
\end{equation} 
\end{theorem}

In order to alleviate the computations, we will simplify the system \eqref{continuous-mEVP complete} into
\begin{equation}
\begin{aligned}
\partial_\tau \bu+ \bu- \operatorname{div}(\sig)&=F(\bu), \\
\partial_\tau \sig+\sig-\sig^{\mathrm{VP}}(\beps)&=0,
\end{aligned}
\label{continuous-mEVP short}
\end{equation}and the VP limit system \eqref{semi-discrete VP limit} into 
\begin{equation}
\begin{aligned}
\bu^* -\operatorname{div}(\sig^*)&=F(\bu^*), \\
\sig^*-\sig^{\mathrm{VP}}(\beps^*)&=0,
\end{aligned}
\label{semi-discrete VP limit short}
\end{equation}
which admit the same properties.

\paragraph{Remark} \textit{At physical time step $t^n$, the external forcings for the momentum equation
\begin{multline}
F(\bu^n) = A^n\big(\rho_a C_a\|\bu_a-\bu^n\|_2(\bu_a-\bu^n)+ \rho_o C_o\|\bu_o-\bu^n\|_2(\bu_o-\bu^n) \big)\\
+ \rho_\text{ice} H f\vec k\times(\bu_o-\bu^n)
\end{multline}
can be approximated by the following semi explicit formulation
\begin{equation}\label{approx forces}
F(\bu^n) \approx A\big (\rho_a C_a\|\bu_a\|_2\bu_a+ \rho_o C_o\|\bu_o-\bu^n\|_2(\bu_o-\bu^n) \big ) 
+ \rho_\text{ice} H f\vec k\times(\bu_o-\bu^{n-1}),
\end{equation}
which falls within the scope of Theorem \ref{thm:VP cv}.}

Before giving the proof of the theorem, we will show a coupled of auxiliary results. 
Lemmas~\ref{lemma:VP potential}-\ref{lemma: VP functional} are essentially also found in Shih et al.~\cite{Shih2023}, where the approach via a minimization problem served as the basis for an efficient Newton method. We present them along with detailed proofs, since the results and the notation introduced form the basis for Theorem~\ref{thm:VP cv}.

\begin{lemma}[VP potential]
\label{lemma:VP potential}
    The viscous plastic (VP) law $\beps \mapsto \sig^{VP}(\beps)$ derives from the VP potential
    \begin{equation}
        \Phi(\beps)=\frac{P}{2}\Delta(\beps)- \frac{P}{2}\operatorname{Tr}(\beps),
    \end{equation}
such that
\begin{equation}
    \nabla_{\beps} \Phi(\beps)=\sig^{VP}(\beps).
\end{equation}
\end{lemma}

\begin{proof}
The differential of $\Delta(\beps)=
\sqrt{\Delta_{\min}^2+\operatorname{Tr}(\beps)^2+2e^{-2}\,\beps':\beps'}$ with respect to $\beps$ writes
\begin{equation*}
D_{\beps} \Delta(\beps)\cdot h=
\frac{\operatorname{Tr}(\beps)\operatorname{Tr}(h)
+2e^{-2}\beps':h'}{\Delta(\beps)}.
\end{equation*}
Considering $\operatorname{Tr}(\beps)\operatorname{Tr}(h)
=\bigl(\operatorname{Tr}(\beps)I\bigr):h$,
since $I:h=\operatorname{Tr}(I^T h)=\operatorname{Tr}(h), $ and again
$
\beps':h'=\beps':h-\frac12 \beps':\bigl(\operatorname{Tr}(h)I\bigr) = \beps':h
$ since $\beps'$ is trace free, we obtain
$D_{\beps} \Delta(\beps)\cdot h
=
\frac{1}{\Delta(\beps)}
\left[
\operatorname{Tr}(\beps)I
+
2e^{-2}\beps'
\right]:h$. 
Now considering the complete potential  $
\Phi(\beps)
=
\frac{P}{2}\Delta(\beps)
-
\frac{P}{2}\operatorname{Tr}(\beps),
$
we have
$
D_{\beps}\left[
\frac{P}{2}\Delta(\beps)
\right]\cdot h=
\frac{P}{2\Delta(\beps)}
\left[
\operatorname{Tr}(\beps)I+2e^{-2}\beps'
\right]:h,
$ and 
$D_{\beps}\left[
\frac{P}{2}\operatorname{Tr}(\beps)
\right]\cdot h=\frac{P}{2}\operatorname{Tr}(h)=
\frac{P}{2}I:h,
$  thus
\begin{equation}
\nabla_{\beps} \Phi(\beps)
=\frac{P}{2\Delta(\beps)}\left[\operatorname{Tr}(\beps)I
+2e^{-2}\beps'\right]-\frac{P}{2}I = \sig^{VP}(\beps).
\end{equation}
\end{proof}

\begin{lemma}[Convexity of $\Phi(\epsilon)$]
\label{lemma: convexity of phi}
    The VP potential $\Phi(\beps)=\frac{P}{2}\Delta(\beps)- \frac{P}{2}\operatorname{Tr}(\beps)$ is a convex function of $\beps$.
\end{lemma}
\begin{proof}
Define the linear map
$$L\beps=
\left(\operatorname{Tr}\beps, \sqrt{2e^{-2}}\beps'\right).
$$
Then 
$\Delta(\beps) = \sqrt{\Delta_\text{min}^2 + \|L\beps\|^2} =
\left\|\left (\Delta_{\min},\|L\beps\| \right )\right\|_2$, with $\| . \|$ the product norm over the space $\mathbb{R} \times\mathbb{M}^2(\mathbb{R})$. Let \(\beps_1,\beps_2\) be two tensors in $\mathbb{M}^2(\mathbb{R})$ and let
\(0\leq \theta\leq 1\). We have
\[
\| L(\theta\beps_1+(1-\theta)\beps_2)\|
\leq
\theta \|L\beps_1\|+(1-\theta)\|L\beps_2 \|.
\]
The function $\varphi : s >0 \mapsto \|(\Delta_{\min}, s)\|_2$ is increasing and positive. Thus 
$$\varphi(\| L(\theta\beps_1+(1-\theta)\beps_2)\|)
\leq
\varphi(\theta \|L\beps_1\|+(1-\theta)\|L\beps_2 \|),
$$
and
$$
\begin{aligned}
\Delta(\theta \beps_1 + (1-\theta)\beps_2) &= 
\left\|\left (\Delta_{\min},\| \theta L\beps_1+(1-\theta)L\beps_2 \| \right )\right\|_2\\
& \leq \left\|\left (\theta \Delta_{\min} + (1-\theta ) \Delta_{\min}, \theta \| L\beps_1 \| +(1-\theta) \| L\beps_2 \| \right )\right\|_2\\
& \leq \theta \Delta (\beps_1) + (1-\theta) \Delta (\beps_2),
\end{aligned}
$$and $\Delta$ is a convex function. Consequently, $\beps \mapsto \Phi(\beps) $ is also convex as a sum of two convex functions.

\end{proof}

\begin{lemma}[Forcings potential]
\label{lemma:Forcing potential}
    The external forcings $\bu \mapsto -F(\bu)$ derive from the convex potential
    \begin{equation}
        \mathcal{V}(\bu)= \frac{1}{3} \|\bu-\bu_o\|_2^3 - f\cdot \bu.
    \end{equation}
\end{lemma} 
\begin{proof}
    Similarly to the previous proofs, we have that $\nabla_{\bu} \mathcal{V}(\bu)  = \|\bu-\bu_o\|_2 (\bu-\bu_o) -f = -F(\bu).$ Moreover, $\mathcal{V}(\bu)$ is a convex function of $\bu$, as a composition of convex functions.
\end{proof}

\begin{lemma}[Viscous-Plastic functional]
\label{lemma: VP functional}
    The solution of the semi-discrete viscous-plastic model \eqref{semi-discrete VP limit short} is the unique minimizer of the strongly convex functional 
    \begin{equation}
        \mathcal{J}(\bu) = \frac12 (\bu,\bu)+\int_\Omega \mathcal{V}(\bu) \,dx + \int_\Omega \Phi(\beps)\,dx,
    \end{equation}with $(.,.)$ the usual scalar product in $L^2(\Omega)$.
\end{lemma}
\begin{proof}
    First, as the sum of a strongly convex and convex functions, the functional $\mathcal{J}$ is also strongly convex. Computing the differential of $\mathcal{J}$ we get 
    $$D\mathcal J(\bu)\cdot h =(\bu,h)+(\nabla_{\bu}\mathcal{V}(\bu),h)+
\bigl(\nabla_{\beps}\Phi(\beps),\nabla^s h\bigr) = (\bu,h)+(-F(\bu),h)+
\bigl(\sig^{VP}(\beps),\nabla^s h\bigr).
    $$Since $\sig^{\mathrm{VP}}(\beps)$ is a symmetric tensor, we obtain the equality
$$
\bigl(\nabla_{\beps}\Phi(\beps),\nabla^s h\bigr)
=
\bigl(\sig^{\mathrm{VP}}(\beps),\nabla h\bigr)
=
\bigl(-\operatorname{div}(\sig^{\mathrm{VP}}(\beps)),h\bigr) + \bigl(\sig^{\mathrm{VP}}(\beps), h\bigr)_{|\partial \Omega}.
$$Considering the homogeneous Dirichlet boundary condition on $h$, we can write
$$D\mathcal J(\bu)\cdot h = \bigl(\bu-\operatorname{div}(\sig^{\mathrm{VP}}(\beps))-F(\bu),h\bigr),
$$and $\nabla \mathcal J(\bu)=\bu-\operatorname{div}(\sig^{\mathrm{VP}}(\beps))-F(\bu),$ with minimum attained in $\bu^*$:$$
\nabla \mathcal J(\bu)=0
\quad\Longleftrightarrow\quad
\bu=\bu^*.$$\end{proof}


\begin{lemma}[mEVP second order system]
\label{lemma: ODE EVP}
The solution $\bu$ of the mEVP system \eqref{continuous-mEVP short} satisfies the second order differential equation
\begin{equation}
    \partial_{\tau \tau}\bu+2\partial_\tau\bu+\nabla\mathcal J( \bu)- D_{\bu}F(\bu)\cdot \partial_\tau \bu=0.
    \label{2nd oder ode mEVP}
\end{equation}
\end{lemma}
\begin{proof}
Considering the mEVP system
\begin{align}
\partial_\tau \bu+ \bu -
\operatorname{div}(\sig)&=F(\bu), \label{eq u 1}\\
\partial_\tau \sig+\sig-\sig^{{VP}}(\beps)&=0,
\label{eq sig 1}
\end{align}
we derive \eqref{eq u 1} with respect to $\tau$:
\begin{equation}\label{second order 1}
\partial_{\tau \tau}\bu+\partial_\tau \bu-
\operatorname{div}(\partial_\tau \sig) - D_{\bu}F(\bu)\cdot \partial_\tau \bu =0.
\end{equation}
From \eqref{eq sig 1}, $\partial_\tau \sig=\sig^{VP}(\beps)-\sig,$ then $\operatorname{div}(\partial_\tau \sig)
=\operatorname{div}(\sig^{VP}(\beps))-
\operatorname{div}(\sig).$ 
Using again \eqref{eq u 1}, we get
$$
\operatorname{div}(\partial_\tau\sig)=\operatorname{div}(\sig^{VP}(\beps))-\partial_\tau \bu-\bu+F(\bu),
$$
which, injected into~\eqref{second order 1} gives
$$\partial_{\tau \tau}\bu+2 \partial_\tau \bu + \bu-
\operatorname{div}(\sig^{VP}(\beps)) - F(\bu) - D_{\bu}F(\bu)\cdot \partial_\tau \bu =0.$$From Lemma \ref{lemma: VP functional}, we obtain the expected result.
\end{proof}
On the discrete level, this correspondence to a second order system has been used by Kimmritz~\cite{kimmritz2015convergence} for the stability analysis of the EVP iteration. 

\begin{lemma}[mEVP Lyapunov function]The function
    \begin{equation}
        \mathcal E(\tau)=\frac12 \|\partial_\tau \bu \|^2+\mathcal J(\bu)-\mathcal J(\bu^*)
    \end{equation}is a Lyapunov function for the mEVP second order PDE \eqref{2nd oder ode mEVP} and the equilibrium point $(\partial_\tau \bu = 0, \bu = \bu^*)$, with $\bu^*$ the solution of the VP system \eqref{semi-discrete VP limit short}.
    \label{lemma: lyapunov fun}
\end{lemma}
\begin{proof}
First, it is immediate to check that $\mathcal{E}\ge 0$ is null for $\bu \equiv \bu^*$ only. 
Testing \eqref{2nd oder ode mEVP} by $\partial_\tau \bu$ we get
\begin{multline*}
(\partial_{\tau \tau} \bu, \partial_\tau \bu ) + 2 \|\partial_\tau \bu\|^2 + (\nabla \mathcal{J}(\bu), \partial_\tau \bu) - (D_{\bu}F(\bu)\cdot \partial_\tau \bu, \partial_\tau \bu) \\
    = \frac{1}{2} \frac{d}{d\tau} \|\partial_\tau \bu\|^2 + 2 \|\partial_\tau \bu\|^2 + \frac{d}{d\tau} \mathcal{J}(\bu) + (\partial_\tau \bu, \nabla^2_{\bu} \mathcal{V}(\bu)\, \partial_\tau \bu) = 0,
\end{multline*}
meaning
    \begin{equation}
        \frac{d}{d\tau} \mathcal{E}(\tau) + 2 \|\partial_\tau
\bu\|^2 + (\partial_\tau \bu, \nabla^2_{\bu} \mathcal{V}(\bu)\, \partial_\tau \bu) = 0.    
\end{equation} The Hessian $\nabla^2_{\bu} \mathcal{V}(\bu)$ is a semi-definite positive matrix:
$$\nabla^2_{\bu} \mathcal{V}(\bu) = \frac{(\bu-\bu_o) \otimes (\bu-\bu_o)}{\|\bu-\bu_o\|_2} + \|\bu-\bu_o\|_2 I,
$$which can be extended by $0 \, I$ in $\bu=\bu_o$. The term $(\partial_\tau \bu, \nabla^2_{\bu} \mathcal{V}(\bu)\, \partial_\tau \bu)$ is thus positive, and $\mathcal{E}$ is a decreasing function of $\tau$: $\forall \tau > 0, \quad \mathcal E(t)\leq \mathcal E(0).$
\end{proof}
\paragraph{Remark} \textit{This result gives the stability of the mEVP solution $\bu$. However it does not imply its convergence. We will then consider a modified Lyapunov function for the PDE \eqref{2nd oder ode mEVP}.}

\begin{lemma}[Control of the Hessian matrix]
\label{lemma: control hessian}
There exists $M>0$ such that
$$
    \bigl(\nabla^2_{\bu}\mathcal V(\bu)(\bu-\bu^*),\bu-\bu^*\bigr)
    \leq
    M\bigl(\mathcal J(\bu)-\mathcal J(\bu^*)\bigr).
$$
\end{lemma}
\begin{proof} Let us define $\psi(\bv)=\frac13 \|\bv\|_2^3$. The forcing potential writes $\mathcal{V}(\bu) = \psi(\bu-\bu_o)- f\cdot \bu$, and $\nabla^2\psi(\bu-\bu_o) = \nabla^2_\bu\mathcal{V}(\bu)$.
From Lemma \ref{lemma: bound Bregman}, we have for $\bv,\, \bv^* \in \mathbb{R}^2$
$$
\psi(\bv)-\psi(\bv^*)-\nabla\psi(\bv^*)\cdot(\bv-\bv^*)
    \geq
    \frac1{12}
    \left (\|\bv\|_2+\|\bv^*\|_2\right )\|\bv -\bv^*\|_2^2.
$$
On the other hand, we have for $h \in \mathbb{R}^2$
\[
\begin{aligned}
    h^T\nabla^2\psi(\bv)h
    &=
    \|\bv\|_2 \|h\|_2^2+\frac{(\bv \cdot h)^2}{\|\bv\|_2}  \\
    &\leq
    2 \|\bv\|_2 \|h\|_2^2 \leq
    2(\|\bv\|_2 + \|\bv^*\|_2)\, \|h\|_2^2.
\end{aligned}
\]
Combining the two estimates and choosing $h = \bv-\bv^*$, we obtain
\[\begin{aligned}
    (\bv-\bv^*)^T\nabla^2\psi(\bv)(\bv-\bv^*) &\leq 24 \left [\psi(\bv)-\psi(\bv^*)-\nabla\psi(\bv^*)\cdot(\bv-\bv^*)\right ] \\
    &\leq M \left [\psi(\bv)-\psi(\bv^*)-\nabla\psi(\bv^*)\cdot(\bv-\bv^*)\right ].
    \end{aligned}
\]
We now apply this inequality in $\bv=\bu-\bu_o$ and $\bv^*=\bu^*-\bu_o$, and integrate over the domain $\Omega$. We get
\[\begin{aligned}
    \bigl(\nabla^2_\bu\mathcal{V}(\bu)(\bu-\bu^*),\bu-\bu^*\bigr)
    &\leq
    M
    \int_\Omega \left [\psi(\bu-\bu_o)-\psi(\bu^*-\bu_o)-\nabla\psi(\bu^*-\bu_o)\cdot(\bu-\bu^*)\right ]\\
    &= M
    \int_\Omega \left [\mathcal{V}(\bu)-\mathcal{V}(\bu^*)-\nabla_{\bu}\mathcal{V}(\bu^*)\cdot(\bu-\bu^*)\right ],
\end{aligned}
\]since $\mathcal{V}(\bu) = \psi(\bu-\bu_o)- f\cdot \bu$ and $\nabla_{\bu}\mathcal{V}(\bu) =\nabla\psi(\bu-\bu_o)-f $.

\noindent As all terms in \(\mathcal J\) are convex, we have 
$$ \mathcal{J}(\bu)-\mathcal{J}(\bu^*)-\nabla \mathcal{J}(\bu^*)\cdot(\bu-\bu^*) \ge 
\mathcal{V}(\bu)-\mathcal{V}(\bu^*)-\nabla_{\bu}\mathcal{V}(\bu^*)\cdot(\bu-\bu^*).$$ And finally, since $\bu^*$ is the unique minimizer of the functional $\mathcal{J}$, $\nabla \mathcal{J}(\bu^*)=0$ and we have
$$ M \left(\mathcal{J}(\bu)-\mathcal{J}(\bu^*)\right) \ge 
\mathcal{V}(\bu)-\mathcal{V}(\bu^*)-\nabla_{\bu}\mathcal{V}(\bu^*)\cdot(\bu-\bu^*)\ge \bigl(\nabla^2_\bu\mathcal{V}(\bu)(\bu-\bu^*),\bu-\bu^*\bigr).$$
\end{proof}

\begin{theorem}[Modified Lyapunov functional]
\label{lemma: modified Lyapunov}
    For $0< \eta <\min(1, \frac{1}{M}) $, the function 
    \begin{equation}
        \mathcal{E}_\eta(\tau) = \mathcal E(\tau)+\eta(\bu-\bu^*,\partial_\tau\bu)+\eta\|\bu-\bu^*\|^2,
    \end{equation}equivalent to $\mathcal{E}(\tau)$, is a strict Lyapunov function for the mEVP second order system \eqref{2nd oder ode mEVP} and the equilibrium point $(\partial_\tau \bu = 0, \bu = \bu^*)$. Moreover, the following inequalities hold
    \begin{equation}
        \mathcal{E}_\eta(\tau) \leq \mathcal{E}_\eta(0)e^{-\frac{\eta}{6}\tau}, \qquad \mathcal{E}(\tau)
\leq 6\, \mathcal{E}(0)e^{-\frac{\eta}{6}\tau}.
    \end{equation}
\end{theorem}
\paragraph{Remark} \textit{In practice, $\frac{\eta}{6}$ is of order $10^{-2}$. }
\begin{proof}\textbf{Equivalence with $\mathcal{E}(\tau)$: }first, we have the lower bound $$
\mathcal{E}_\eta(\tau) \ge \mathcal E(\tau)-\frac{1}{2}\eta \left [\|\bu-\bu^*\|^2 + \|\partial_\tau\bu\|^2 \right ]+\eta\|\bu-\bu^*\|^2 \ge \frac{1}{2 }\mathcal{E}(\tau) + \frac{\eta}{2}\|\bu-\bu^*\|^2\ge \frac{1}{2 }\mathcal{E}(\tau),
$$implying that $\mathcal{E}_\eta(\tau) $ is also strictly positive for $\bu \ne \bu^*$, and that $\mathcal{E}_\eta(\bu^*,\tau)=0$. Second, using Cauchy-Schwarz and Young inequalities, we can write
$$
\mathcal E(\tau)+\eta(\bu-\bu^*,\partial_\tau\bu)+\eta\|\bu-\bu^*\|^2 \leq \mathcal E(\tau) + \frac{3}{2}\eta \|\bu - \bu^*\|^2 + \frac{\eta}{2} \|\partial_\tau \bu\|^2.
$$The strong convexity of $\mathcal{J}$, whose minimum is attained in $\bu^*$ gives $\mathcal{J}(\bu) - \mathcal{J}(\bu^*) \ge \|\bu - \bu^*\|^2$, thus, the following inequality holds
$$\mathcal E_\eta(\tau) \leq \mathcal E(\tau) + \frac{3}{2}\eta \left (\mathcal{J}(\bu)-\mathcal{J}(\bu^*)  \right ) + \frac{\eta}{2} \|\partial_\tau \bu\|^2 \leq 3 \, \mathcal E(\tau),
$$and completes the equivalence relationship.

\noindent \textbf{Lyapunov property: } from Lemma \ref{lemma: lyapunov fun}, we have
$$
\frac{d}{d\tau}\mathcal{E}_\eta(\tau)  = -2 \|\partial_\tau \bu\|^2 + \eta (\partial_\tau \bu, \partial_\tau \bu) + \eta (\partial_{\tau\tau}\bu, \bu-\bu^*) + 2 \eta (\partial_{\tau}\bu, \bu-\bu^*).
$$Using the second order system \eqref{2nd oder ode mEVP} to replace $\partial_{\tau\tau}\bu$, we get
\begin{equation}
\label{inter dEeta}
    \begin{aligned}
    \frac{d}{d\tau}\mathcal{E}_\eta(\tau)  = &  -(2-\eta) \|\partial_\tau \bu\|^2  - \eta (\nabla \mathcal{J}(\bu), \bu-\bu^*) \\
    &\qquad -(\partial_\tau \bu, \nabla^2_{\bu} \mathcal{V}(\bu)\, \partial_\tau \bu)- \eta (\nabla^2_{\bu} \mathcal{V}(\bu)\,  \partial_\tau \bu , \bu-\bu^*).
\end{aligned}
\end{equation}
Since $\nabla^2_{\bu} \mathcal{V}(\bu)\succeq 0 $ and as it is symmetric, there exists a symmetric ``square root'' matrix $R(\bu)$ such that $R(\bu)^T R(\bu) = \nabla^2_{\bu} \mathcal{V}(\bu)$. Applying Cauchy-Schwarz inequality followed by Young inequality, we have 
$$\begin{aligned}
    -\eta (\nabla^2_{\bu} \mathcal{V}(\bu)\,  \partial_\tau \bu , \bu-\bu^*) &\leq \eta \|R(\bu)\partial_\tau \bu\| \|R(\bu)(\bu-\bu^*) \|\\
    &\leq \frac{1 }{2 }\|R(\bu)\partial_\tau \bu\|^2 +  \frac{\eta^2 }{2}\|\nabla^2_{\bu} \mathcal{V}(\bu)\|\|\bu-\bu^*\|^2.
\end{aligned}$$ Inserting this inequality in \eqref{inter dEeta} and using Lemma \ref{lemma: control hessian}, we obtain
\begin{equation*}
    \begin{aligned}
    \frac{d}{d\tau}\mathcal{E}_\eta(\tau)  \leq &  -(2-\eta) \|\partial_\tau \bu\|^2  - \eta (\nabla \mathcal{J}(\bu), \bu-\bu^*) \\
    &\quad -\frac12\|R(\bu)\partial_\tau \bu\|^2 +  \frac{\eta^2}{2} M (\mathcal{J}(\bu)-\mathcal{J}(\bu^*)).
\end{aligned}
\end{equation*}
From the convexity of $\mathcal{J}$, the inequality 
$$(\nabla \mathcal{J}(\bu), \bu^*-\bu)\leq\mathcal{J}(\bu^*)-\mathcal{J}(\bu) \Leftrightarrow -(\nabla \mathcal{J}(\bu), \bu-\bu^*)\leq - \left [\mathcal{J}(\bu)-\mathcal{J}(\bu^*)\right]
$$holds. Then, we can obtain the following bound
$$
\begin{aligned}
    \frac{d}{d\tau}\mathcal{E}_\eta(\tau)  &\leq -(2-\eta) \|\partial_\tau \bu\|^2 - \frac{\eta}{2} (\mathcal{J}(\bu) - \mathcal{J}(\bu^*) ) -\frac{1 }{2 }\|R(\bu)\partial_\tau \bu\|^2 - \frac{\eta}{2} \left[ 1-\eta M  \right]\|\bu-\bu^*\|^2\\
    & \leq -\frac{\eta}{2} \mathcal{E}(\tau) \leq -\frac{\eta}{6} \mathcal{E}_\eta(\tau),
\end{aligned}$$ 
proving that $\mathcal{E}_\eta(\tau)$ is a strict Lyapunov function for \eqref{2nd oder ode mEVP}. 
From Gronwall's lemma, we further obtain 
$$
\mathcal{E}_\eta(\tau)
\leq \mathcal{E}_\eta(0)e^{-\frac{\eta}{6}\tau}.
$$ and
$$
\mathcal{E}(\tau)
\leq 2 \mathcal{E}_\eta(0)e^{-\frac{\eta}{6}\tau}.
    $$
\end{proof}

\begin{proof}[\textbf{Proof of Theorem \ref{thm:VP cv}}]

\noindent  \textbf{$L^2$-convergence for $\bu$.} From Lemma \ref{lemma: modified Lyapunov}, the trajectories of \eqref{2nd oder ode mEVP} satisfy
$$\frac{1}{4} \|\partial_\tau\bu\|^2 + \frac{\eta}{2}  \|\bu-\bu^*\|^2 \leq \mathcal{E}_\eta (0)e^{-\frac{\eta}{6} \tau},
$$which directly gives the $L^2$-convergence of $\bu$ towards $\bu^*$.\\

\noindent \textbf{$\sig^{VP}(\beps)$-convergence.} The VP potential $\Phi(\beps)$ is a convex function whose gradient $\nabla_{\beps}\Phi(\beps) = \sig^{VP}(\beps)$ is Lipschitz:
$$\|\sig^{VP}(\beps)-\sig^{VP}(\beps^*)\| \leq \frac{\|P\|_\infty}{\Delta_{\min}} \| \text{Tr} (\beps-\beps^*) I + 2 e^{-2}(\beps-\beps^*)'\| \leq \frac{6 \|P\|_\infty}{\Delta_{\min}}  \|\beps-\beps^*\|.
$$Thus, denoting $L$ the Lipschitz constant, the following classical inequality holds
\begin{equation}
    \Phi(\beps^*) - \Phi(\beps) \leq \sig^{VP}(\beps^*):(\beps^*-\beps)-\frac{1}{2L}\|\sig^{VP}(\beps)-\sig^{VP}(\beps^*)\|_2^2.
    \label{inter maj sigVP}
\end{equation}Equivalently we have
\begin{equation}
    \frac{1}{2L}\|\sig^{VP}(\beps)-\sig^{VP}(\beps^*)\|_2^2 \leq \Phi(\beps) - \Phi(\beps^*) - \sig^{VP}(\beps^*):(\beps-\beps^*),
    \label{inter maj sigVP}
\end{equation}
which can be used to control $\|\sig^{VP}(\beps)-\sig^{VP}(\beps^*)\|_2^2$.
Noticing that, for all test function $h \in (H^1_0(\Omega))^2$ the variational formulation of \eqref{semi-discrete VP limit short} give $(\bu^*-f, h) = (-\sig^*, \nabla^s h) $, we have
$$
\begin{aligned}
    \mathcal{J}(\bu)-\mathcal{J}(\bu^*) &=  \frac12 (\bu,\bu)-\frac12 (\bu^*,\bu^*)-(f,\bu-\bu^*)+ \int_\Omega \Phi(\beps)-\Phi(\beps^*)\\
    &= \frac12 \|\bu-\bu^*\|^2 + (\bu^*-f, \bu-\bu^*)+ \int_\Omega \Phi(\beps)-\Phi(\beps^*)\\
    &= \frac12 \|\bu-\bu^*\|^2 - (\sig^*, \beps-\beps^*)+ \int_\Omega \Phi(\beps)-\Phi(\beps^*)\\
    &\ge \int_\Omega \Phi(\beps)-\Phi(\beps^*)- (\sig^*, \beps-\beps^*) ,
\end{aligned}
$$which can be inserted in \eqref{inter maj sigVP} after integration over space to obtain
\begin{equation*}
    \frac{1}{2L}\|\sig^{VP}(\beps)-\sig^{VP}(\beps^*)\|^2 \leq \mathcal{J}(\bu) - \mathcal{J}(\bu^*)\leq \mathcal{E}_\eta (0)e^{-\frac{\eta}{6}  \tau},
\end{equation*}proving the $L^2$-convergence of $\sig^{VP}(\beps)$ towards $\sig^{VP}(\beps^*) = \sig^*$.\\

\noindent \textbf{$\sig$-convergence.} We now prove the $H^{\operatorname{div}}(\Omega)$ convergence of $\sig$. First, substracting $\sig$ and $\sig^*$ equations, and applying the resulting equation on the test function $\sig-\sig^*$, we get
\begin{multline*}
\frac12 \frac{d}{d\tau} \|\sig-\sig^*\|^2 + \|\sig-\sig^*\|^2 = (\sig^{VP}(\beps)-\sig^{VP}(\beps^*), \sig-\sig^*)
\\
\leq \frac{1}{2}\|\sig^{VP}(\beps)-\sig^{VP}(\beps^*)\|^2+\frac{1}{2}\|\sig-\sig^*\|^2,
\end{multline*}
which, from previous results, gives the $L^2$-convergence of both $\sig$ and its temporal derivative with an exponential rate:
$$
\frac12 \frac{d}{d\tau} \|\sig-\sig^*\|^2 + \frac{1}{2}\|\sig-\sig^*\|^2 \leq 2L\mathcal{E}_\eta (0)e^{-\frac{\eta}{6}  \tau}.
$$ Note that $L$ can still be very big, as it is scaled by the viscosity $1/\Delta_{\min} \approx 10^9$.
The convergence of $\text{div}(\sig)$ is immediately deduced from the convergence of the viscosity $\bu$ : $\|\text{div}(\sig)-\text{div}(\sig^*)\| \leq \|\partial_\tau \bu\| + \|\bu-\bu^*\|\leq C e^{-\frac{\eta}{6}  \tau } $, with $C > 0$. Finally, we have 
$$\underset{\tau \rightarrow\infty}{\lim }\|\sig-\sig^*\|^2_{H^{\operatorname{div}}(\Omega)} =0,
$$which ends the proof.
\end{proof}

\paragraph{Remark} \textit{The convergence of both $\bu$ and $\sig$ does not necessarily imply the convergence of the strain tensor $\beps$. Indeed, due to the lack of coercivity, and with not additional hypothesis, the control of the VP potential, or of the VP law, is not sufficient to bound $\beps$.}

\textit{In the hypothesis of a lower bound for the viscosity $\frac{1}{\Delta(\beps)}$, a bound on $\|\beps\|$ can be established \cite{Mehlmann2021}, leading to the $L^2$-weak convergence of $\beps$. Then, reusing some inequalities from the proof, $L^2$-strong convergence of $\beps$ can be achieved.}


\section{A mixed form DG discretization for the sea ice model}
This section presents the discontinuous Galerkin sea ice dynamics model, based on a Local Discontinuous Galerkin (LDG) scheme for the mEVP equation. 
\subsection{LDG scheme for the inner mEVP system}
We describe here the discretization of the mEVP inner system using a Local Discontinuous Galerkin scheme \cite{di2011mathematical, cockburn1998local}. We consider the simplified pseudo-time mEVP mometum equation written as a first-order system on the domain $\Omega$
\begin{align}
 \partial_\tau \mathbf{u} +\bu - \nabla \cdot \boldsymbol{\sigma} &= F(\mathbf{u}), \label{eq u}\\
\partial_\tau \boldsymbol{\sigma} + \boldsymbol{\sigma}
&= \boldsymbol{\sigma}^{VP}(\beps), \label{eq sig} \\
\beps- \frac{1}{2} (\nabla \mathbf{u} + \nabla \mathbf{u} ^T) &=0 \label{eq eps}.
\end{align}
Let $\mathcal{T}_h = \bigcup_K K$ be a triangular or quadrilateral mesh of the domain $\Omega$ of mesh size $h$, 
$V^n(\mathcal{T}_h) = \left \{ \bu \in L^2(\mathcal{T}_h)^2\, | \, \bu_{|K}  \in (\mathcal{O}^n(K))^2 \right \}$ --- with $\mathcal{O}^n(K):=\mathcal{P}^n(K)$ for triangular mesh and $\mathcal{O}^n(K):=\mathcal{Q}^n(K)$ for quadrilateral mesh --- the DG space of velocity vectors of order $n>0$, and $W^m(\mathcal{T}_h) = \left \{ \sig \in L^2(\mathcal{T}_h)^{2\times 2} \, | \, \sig^T = \sig,\,  \sig_{|K}  \in (\mathcal{O}^m(K))^{2\times 2} \right \}$ the DG space of symmetric matrices of order $m\ge0$ on the discretization $\mathcal{T}_h$. We consider the following trial functions 
\[
\bu_h \in V^n(\mathcal{T}_h), \qquad
\beps_h \in W^m(\mathcal{T}_h), \qquad
\sig_h \in W^m(\mathcal{T}_h),
\] 
and test functions
\[
\bv_h \in V^n(\mathcal{T}_h), \qquad
\btau_h \in W^m(\mathcal{T}_h), \qquad
\bt_h \in W^m(\mathcal{T}_h),
\]
for the three equations \eqref{eq u}, \eqref{eq sig} and \eqref{eq eps} respectively. 
Multiplying these by the appropriate test function and integrating over a mesh element $K$, we obtain, $\forall \bv_h \in V^n(\mathcal{T}_h) $ and $\btau_h \in W^m(\mathcal{T}_h)$,
\begin{align}
 ( \partial_\tau \mathbf{u}_h, \mathbf{v}_h)_K + ( \mathbf{u}_h, \mathbf{v}_h)_K+ (\boldsymbol{\sigma}_h, \nabla \mathbf{v}_h)_K-
\langle \widehat{\boldsymbol{\sigma}}_h  \mathbf{n}, \mathbf{v}_h \rangle_{\partial K}& = ( F(\mathbf{u}_h), \mathbf{v}_h)_K, \\
 (\partial_\tau \boldsymbol{\sigma}_h, \bt_h)_K
+
(\boldsymbol{\sigma}_h, \bt_h)_K
&= (\boldsymbol{\sigma}^{VP}(\mathbf{\epsilon}_h), \bt_h)_K,  \\
   (\mathbf{\beps}_h, \boldsymbol{\tau}_h)_K
+(\mathbf{u}_h, \nabla \cdot \boldsymbol{\tau}_h)_K-
\langle \widehat{\mathbf{u}}_h, \boldsymbol{\tau}_h \cdot \mathbf{n} \rangle_{\partial K}&=0,
\end{align}
where $(.,.)$ denotes the usual scalar product in $(L^2(\Omega))^2$ or $(L^2(\Omega))^{2\times2}$, the $\hat{.}$ symbol indicates numerical fluxes on the element boundary, and $n$ is the outward normal to the boundary. Summing over the elements of the mesh, the semi-discrete first order system writes 
\begin{align}
 \sum_{K} ( \partial_\tau \mathbf{u}_h, \mathbf{v}_h)_K+ \sum_{K} (\mathbf{u}_h, \mathbf{v}_h)_K
+ \sum_{K} (\boldsymbol{\sigma}_h, \nabla \mathbf{v}_h)_K
\qquad \nonumber
\\
-\sum_{F \in \mathcal{F}_h^i} \langle \widehat{\boldsymbol{\sigma}}_h \mathbf{n}, \ljp \mathbf{v}_h  \rjp\rangle_{F} -\sum_{F \in \mathcal{F}_h^b} \langle \widehat{\boldsymbol{\sigma}}_h \mathbf{n}, \mathbf{v}_h \rangle_{F}
&= \sum_{K} (F(\bu_h), \mathbf{v}_h)_K, \label{u line}
\\
 \sum_K (\partial_\tau \boldsymbol{\sigma}_h, \bt_h)_K
+
\sum_K (\boldsymbol{\sigma}_h, \bt_h)_K
&= \sum_K (\boldsymbol{\sigma}^{VP}(\mathbf{\epsilon}_h), \bt_h)_K, \label{sig line}\\
 \sum_{K} (\beps_h, \boldsymbol{\tau}_h)_K
+ \sum_{K} (\bu_h, \nabla \cdot \boldsymbol{\tau}_h)_K \qquad\qquad\nonumber\\
- \sum_{F \in \mathcal{F}_h^i} \langle \widehat{\mathbf{u}}_h, \ljp \boldsymbol{\tau}_h \mathbf{n} \rjp \rangle_{F} - \sum_{F \in \mathcal{F}_h^b} \langle \widehat{\mathbf{u}}_h,  \boldsymbol{\tau}_h \mathbf{n} \rangle_{F}&=0, \label{eps line}
\end{align}
with $F\in \mathcal{F}_h^i$ the interior faces of the mesh and $F\in \mathcal{F}_h^b$ its boundary faces. 
We choose LDG \cite{cockburn1998local, di2011mathematical} numerical fluxes such that on interior faces
\begin{align}
\widehat{\mathbf{u}} &= \{\!\{\bu \}\!\} + a \ljp \bu\rjp, \\
\widehat{\sig} n
&=
\{\!\{\sig n\}\!\} - a \ljp \sig n \rjp
- b [\![\bu]\!],
\end{align}where $\{\!\{\cdot\}\!\}$ and $[\![\cdot]\!]$ denote averages and jumps across faces, and flux parameters $(a, b)$ are such that $(a, b) \in [0, 0.5)\times (0, +\infty)$. On the boundary, we impose a classical no slip condition on the velocity and, considering a neighboring ghost value, we impose a transparent no jump condition on $\sig n$, compatible with the interior fluxes. The boundary fluxes write
\begin{align}
\widehat{\mathbf{u}} &= 0 \, (= \{\!\{\bu \}\!\} + a \ljp \bu\rjp), \\
\widehat{\sig} n
&= \sig n - \frac{b}{0.5-a} \bu \, (= 
\{\!\{\sig n\}\!\} - a \ljp \sig n \rjp
- b [\![\bu]\!]).
\end{align}
Now rewriting the $\beps$ line to avoid the divergence of $\btau$ term, and using the above LDG fluxes, the complete semi-discretized system writes
\begingroup
\small
\begin{subequations}
    \label{semi discrete mEVP}
\begin{align}
 \sum_{K} ( \partial_\tau \mathbf{u}_h, \mathbf{v}_h)_K+ \sum_{K} (\mathbf{u}_h, \mathbf{v}_h)_K
+ \sum_{K} (\boldsymbol{\sigma}_h, \nabla \mathbf{v}_h)_K\qquad &\nonumber \\
-\sum_{F \in \mathcal{F}_h^i} \langle \{\!\{\sig n\}\!\} - a \ljp \sig n \rjp
- b [\![\bu]\!]), \ljp \mathbf{v}_h  \rjp\rangle_{F} 
%
-\sum_{F \in \mathcal{F}_h^b} \langle \boldsymbol{\sigma}_h \mathbf{n} - \frac{b}{0.5-a} \bu , \mathbf{v}_h \rangle_{F}
&= \sum_{K} (F(\bu_h), \mathbf{v}_h)_K,
\label{u line sd}\\
 \sum_K (\partial_\tau \boldsymbol{\sigma}_h, \bt_h)_K
+
\sum_K (\boldsymbol{\sigma}_h, \bt_h)_K
&= \sum_K (\boldsymbol{\sigma}^{VP}(\mathbf{\epsilon}_h), \bt_h)_K, \label{sig line sd}\\
\sum_{K} (\beps_h, \boldsymbol{\tau}_h)_K - \sum_{K} (\nabla \mathbf{u}_h, \boldsymbol{\tau}_h)_K  
\qquad&\nonumber
\\
- \sum_{F \in \mathcal{F}_h^i}  \left[ \langle \{\!\{\bu \}\!\} + a \ljp \bu\rjp), \ljp \boldsymbol{\tau}_h \mathbf{n} \rjp \rangle_F - \ljp \langle \mathbf{u}_h, \boldsymbol{\tau}_h \mathbf{n} \rangle_F \rjp \right] + \sum_{F \in \mathcal{F}_h^b}   \langle \mathbf{u}_h, \boldsymbol{\tau}_h \mathbf{n} \rangle_F& =0. \label{eps line sd}
\end{align}
\end{subequations}
\endgroup

\paragraph{Stability of the semi-discretized system for specific linear rheologies}
Proposition \ref{prop:stability} provides an $L^2$-energy estimate for the velocity and stress variables of the semi-discrete mEVP system \eqref{semi discrete mEVP}, under the assumption of a specific linear constitutive relation between strain and stress tensors: $\sig^{VP}(\beps) = L\beps$ with $L$ a symmetric definite positive matrix. It shows in this simplified setting that the LDG spatial discretization does not introduce additional instabilities into the mEVP system, owing to the compensation of the chosen numerical fluxes.

\paragraph{Convergence and stabilization}
LDG convergence orders have been studied for linear elliptic equations in \cite{castillo2000priori} and depend, in particular, on the mesh-size scaling of the stabilization parameter associated with the velocity jump $b$. To recover optimal convergence for the primal variable, $b$ should be scaled as $1/h$ \cite{castillo2000priori}, where $h$ is the mesh size. In our case, this stabilization plays a second role, the auxiliary variable being the strain tensor rather than the full gradient of the primal variable. Indeed, the broken strain norm alone does not control the broken gradient because of elementwise rigid-body modes, and a discrete Korn inequality therefore requires additional jump control \cite{hansbo2003discontinuous}. The velocity-jump stabilization, with $b \propto \frac{1}{h}$, provides this control. A similar stabilization is discussed in more detail in \cite{Mehlmann2021} for Crouzeix-Raviart elements discretization of the mEVP model.

Compatibility of the velocity $V^n(\mathcal{T}_h)$ and stress and strain $W^m(\mathcal{T}_h)$ spaces must also be ensured (see \cite{richter2023dynamical}). In particular, the following inclusion must be satisfied
$$\frac12 \left(\nabla \bv_{h} + \nabla \bv_{h} ^T \right) \in W^m(\mathcal{T}_h)
$$ for all $\bv_h \in V^n(\mathcal{T}_h)$. In practice, we will use the $\left (V^1(\mathcal{T}_h), W^0(\mathcal{T}_h)\right )$ and $\left (V^2(\mathcal{T}_h), W^1(\mathcal{T}_h) \right )$ couples of spaces for triangular meshes, and the couples  $\left (V^1(\mathcal{T}_h), W^1(\mathcal{T}_h)   \right )$ and $\left (V^2(\mathcal{T}_h), W^2(\mathcal{T}_h) \right )$ for quadrilateral meshes. \label{sec:compatibility}


\subsection{Fully discontinuous sea-ice dynamics}

\paragraph{DG-IMEX mEVP discrete system}
In the complete sea-ice momentum equation, both physical and pseudo time discretizations come into play. We denote the physical time steps by $t^n=t^0 + n \Delta t$ and the associated time discrete solution $(\bu^n, \sig^n)$. Pseudo time steps are denoted as $\tau^k = k \, \Delta \tau$ and the associated solutions $(\bu^k, \sig^k)$. We apply an implicit-explicit (IMEX) first order Euler temporal discretization in pseudo time, upon an implicit Euler scheme in physical time for the mEVP equation. The time-discretized system writes
\begin{equation}
\begin{aligned}
  \frac{1}{\Delta \tau}\,\rho_\text{ice}H(\bu^{k+1} - \bu^k)
  + \rho_{\mathrm{ice}} H\,(\bu^{k+1} - \bu^n)
  - \Delta t\, \nabla \cdot \sig^{k+1} &= \Delta t\, F(\bu^k),\\
  \frac{1}{\Delta \tau}\,(\sig^{k+1} - \sig^k)
  + \sig^{k+1}
  &= \sig^{VP}(\beps^k).
  \end{aligned}
  \label{eq:mevp_momentum_full}
\end{equation}
The mixed form DG spatial discretization is then applied, to form the fully discrete system:
\begingroup
\small
\begin{subequations}
\begin{align}
\sum_{K} ( \rho_{\mathrm{ice}} H\,(\frac{1}{\Delta \tau}+1)\, \mathbf{u}_h^{k+1},\mathbf{v}_h)_K+ \Delta t \sum_{K}& (\boldsymbol{\sigma}_h^{k+1}, \nabla \mathbf{v}_h)_K -\Delta t\sum_{F \in \mathcal{F}_h^i} \langle \widehat{\boldsymbol{\sigma}}_h^{k+1} \mathbf{n}, \ljp \mathbf{v}_h  \rjp\rangle_{F}\nonumber \\
-\Delta t\sum_{F \in \mathcal{F}_h^b} \langle \widehat{\boldsymbol{\sigma}}_h^{k+1} \mathbf{n}, \mathbf{v}_h \rangle_{F}
= \Delta t \sum_{K} &(F(\bu_h^{k}), \mathbf{v}_h)_K + \sum_{K} (\rho_{\mathrm{ice}} H\,( \frac{1}{\Delta \tau} \mathbf{u}_h^{k} + \bu_h^n), \mathbf{v}_h)_K, \label{u line}
\\
\sum_K ((\frac{1}{\Delta\tau}+1) \,\boldsymbol{\sigma}_h^{k+1}, \bt_h)_K
&= \sum_K (\boldsymbol{\sigma}^{VP}(\mathbf{\epsilon}_h^{k}), \bt_h)_K + \sum_K (\frac{1}{\Delta\tau} \,\boldsymbol{\sigma}_h^{k}, \bt_h)_K,\label{sig line}\\
\sum_{K} (\beps_h^{k+1}, \boldsymbol{\tau}_h)_K
+ \sum_{K} (\bu_h^{k+1}, \nabla \cdot \boldsymbol{\tau}_h)_K 
&- \sum_{F \in \mathcal{F}_h^i} \langle \widehat{\mathbf{u}}_h^{k+1}, \ljp \boldsymbol{\tau}_h \mathbf{n} \rjp \rangle_{F} - \sum_{F \in \mathcal{F}_h^b} \langle \widehat{\mathbf{u}}_h^{k+1},  \boldsymbol{\tau}_h \mathbf{n} \rangle_{F}=0, \label{eps line}
\end{align}
\label{discrete mEVP system}
\end{subequations}
\endgroup%
using the formerly introduced LDG fluxes.
Note that the mixed-form DG discretization is applied implicitly, while the forcings and the VP rheology are applied explicitly in pseudo-time $\tau$, analogously to standard mEVP solvers \cite{richter2023dynamical}.
A numerical analysis of LDG-IMEX schemes for convection diffusion equations in the case of Runge-Kutta time discretization can be found in \cite{wang2016local}.

Due to the potentially very large viscosity values $\eta, \zeta \propto \frac{1}{\Delta (\beps)} \in [0,\, 5\cdot  10^{8}]$ arising in the viscous-plastic constitutive law $\sig^{VP}(\beps)$, the LDG jump penalty parameter $b$ must be scaled with the maximal viscosity in order to ensure numerical stability. By contrast, the alternating parameter $a$ was found to have only a negligible influence on the numerical solution. In practice, we set $a = 0.4$ and $b = \frac{10^9}{h}$, where $h$ denotes the characteristic mesh size. The numerical fluxes are then given by
\begin{subequations}
    \begin{align}
\widehat{\mathbf{u}} &= \{\!\{\bu \}\!\} + 0.4 \, \ljp \bu\rjp, \\
\widehat{\sig} n
&=
\{\!\{\sig n\}\!\} - 0.4 \ljp \sig n \rjp
- \frac{10^9}{h} [\![\bu]\!].
\end{align}
\label{seq:numerical fluxes}
\end{subequations}

\paragraph{Remark: solution of the discrete system} 
\textit{Each sub-iteration step of the fully discrete LDG formulation of the momentum equation~\eqref{discrete mEVP system} takes the form
\begin{equation}\label{fd:solve:1}
    \mathbb{A} \begin{pmatrix}
        \bu_h^{k+1} \\
        \sig_h^{k+1}\\
        \beps_h^{k+1}
    \end{pmatrix} = \mathbb{F},\text{ where}\quad
    \mathbb{A} = \begin{pmatrix}
    M_{\bu} + S_{\bu} & G_{\bu} & 0 \\
    0& M_{\sig} & 0\\
    G_{\beps} & 0 & M_{\beps}
    \end{pmatrix},\quad
    \mathbb{F} = \begin{pmatrix}
        F_\bu \\ F_{\sig} \\ F_{\beps}
    \end{pmatrix}.
\end{equation}
By $M_\bu$, $M_{\sig}$ and $M_{\beps}$ we denote mass matrices in the corresponding discrete spaces. They are block-diagonal and local on every element of the mesh. $G_\bu$ and $G_{\beps}$ are discretizations of the gradient and, in particular, they contain fluxes and thus introduce couplings to their directly adjacent elements. $S_\bu$ is the $\bu$-dependence of the $\hat\sig$-flux that also coupled neighboring elements -- see~\eqref{seq:numerical fluxes}. These matrices are multiplied with coefficients that depend on the tracers $A$ and $H$ as well as the time step, the mEVP parameters and several material constants. These coefficients are not relevant for the numerical solution, \eqref{eq:mevp_momentum} gives details.}

\textit{In this work, which was not focused on efficiency, \eqref{fd:solve:1} was inverted as a global system. However, the computational cost of this inversion can be alleviated. Equation~\eqref{fd:solve:1} decouples into three substeps than can be solved subsequently 
\begin{equation}\label{fd:solve:3}
\begin{aligned}
(a)&&M_{\sig}\sig_h^{k+1} &= F_{\sig}\\
(b)&&(M_\bu+S_\bu) \bu_h^{k+1} &= F_\bu - G_\bu \sig_h^{k+1}\\
(c)&&M_{\beps} \beps_h^{k+1} &= F_{\beps} - G_{\beps} \bu_h^{k+1}.
\end{aligned}
\end{equation}
The three sub-steps can be solved efficiently each. $(a)$ and $(c)$ only involve the inversion of the mass matrix without couplings between elements and can be run fully in parallel. Using orthogonal basis functions, even the inversion of the small local element mass matrices could be avoided. Step $(b)$ includes the jump terms~\eqref{seq:numerical fluxes} that behave like a discretization of the Laplacian. No coupling between the two velocity components is introduced such that the problem corresponds to two scalar heat-equation like systems which can be efficiently inverted using a preconditioned CG method or, in linear complexity, using a multigrid iteration~\cite{MehlmannRichter2017mg}.}

\paragraph{DG - Runge Kutta scheme for the tracers $A$ and $H$}
The transport equations for the tracers $A$ and $H$
\begin{align}
    \partial_t A + \operatorname{div}(\bu A) &= 0, & A &\in [0,1],\\
    \partial_t H + \operatorname{div}(\bu H) &= 0, & H &\ge 0,
\end{align}
are computed following \cite{richter2023dynamical} with a 0-order DG upwind space scheme combined with an order 2 Runge-Kutta time scheme. A null-flux boundary condition is applied, following the homogeneous Dirichlet boundary condition for the velocity $\bu$.

We denote by $S^p(\mathcal{T}_h) = \left \{ s \in L^2(\mathcal{T}_h)\, | \, s_{|K}  \in \mathcal{P}^n(K) \right \}$ the $p$-th order discontinuous Galerkin space of tracers on the mesh $\mathcal{T}_h$. For $s_h \in S^p(\mathcal{T}_h)$ and a continuous velocity $\bu$ the upwind DG formulation for A (or H) reads
\begin{equation}
    \begin{aligned}
        \sum_K \left [ \left ( \partial_t A_h, s_h \right )_K -  \left( A_h \bu, s_h \right )_K \right ]  + \sum_{F\in \mathcal{F}_h^i} \left [ \left < \ljp A_h \rjp , \bu \cdot n \, \ljp s_h \rjp  \right > + \frac12 \left < |\bu \cdot n| \, \ljp A_h \rjp , \ljp s_h \rjp  \right > \right ]=0.
    \end{aligned}
\end{equation}
In our case, the discrete velocity $\bu_h$ also belongs to a discontinuous space, so the terms $\bu \cdot n$ are replaced by the LDG flux terms $\hat{\bu}_h \cdot n $, and the upwind discretization reads
\begin{equation}
    \begin{aligned}
        \sum_K \left [ \left ( \partial_t A_h, s_h \right )_K -  \left( A_h \bu_h, s_h \right )_K \right ]  + \sum_{F\in \mathcal{F}_h^i} \left [ \left < \ljp A_h \rjp , \hat{\bu}_h \cdot n \, \ljp s_h \rjp  \right > + \frac12 \left < |\hat{\bu}_h \cdot n| \, \ljp A_h \rjp , \ljp s_h \rjp  \right > \right ]=0.
    \end{aligned}
    \label{upwind A}
\end{equation} Now denoting $\mathbb{C}_{\bu_h}$ and $\mathbb{M}_S $ the advection and mass matrices over $S^p(\mathcal{T}_h)$ such that the scheme \eqref{upwind A} writes $\mathbb{M}_S \, \partial_t A_h + \mathbb{C}_{\bu} A_h = 0$, the fully discrete DG-RK2 scheme for the tracers write
\begin{align}
    &A_h^{n+1}=A_h^n+\Delta t\,\mathbb{M}_S^{-1} \mathbb{C}_{\bu_h^{n+1}}\left(A_h^n+\frac{\Delta t}{2}\mathbb{M}_S^{-1} \mathbb{C}_{\bu_h^{n+1}} A_h^n\right),\\
    &H_h^{n+1}=H_h^n+\Delta t\,\mathbb{M}_S^{-1} \mathbb{C}_{\bu_h^{n+1}}\left(H_h^n+\frac{\Delta t}{2}\mathbb{M}_S^{-1} \mathbb{C}_{\bu_h^{n+1}} H_h^n\right),
\end{align}
with the index $n$ still referring to the physical time stepping. After advecting the tracers, limiters are applied to satisfy $H\ge 0$ and $A \in [0,1]$ conditions.

\paragraph{Remark} \textit{In this work, we use 0-order discontinuous Galerkin spaces for the tracers. The order for the tracers space can of course be increased, as done in \cite{richter2023dynamical}, but was not the focus of this study. Moreover authors of \cite{richter2023dynamical} have shown that the choice of velocity and stress discretization played a more important role in resolving linear kinematic features for the sea ice dynamics than that of the tracers.}

The current global DG sea ice dynamics computation is summarized in Algorithm \ref{alg:ldg_mevp_short}.
\begin{algorithm}[h]
\caption{DG sea-ice dynamics}
\label{alg:ldg_mevp_short}
\begin{algorithmic}[1]
\State Initialize $\bu^0,\beps^0,\sig^0,A^0,H^0$

\For{$n=0,\dots,N_t-1$}
    \State Update forcings
    \State Assemble once + precondition the LDG--mEVP operator $\mathbb{A}(A^n,H^n)$

    \For{$k=0,\dots,N_{\mathrm{mEVP}}-1$}
        \State Recompute forcings $F(\bu_h^{k},A_h^n,H_h^n)$ and
        $\sig^{\mathrm{VP}}(\beps_h^{k},A_h^n,H_h^n)$
        \State Solve the linear mEVP system for
        $(\bu_h^{k+1},\beps_h^{k+1},\sig_h^{k+1})$
    \EndFor

    \State Set
    $(\bu^{n+1},\beps^{n+1},\sig^{n+1})
    =(\bu^{N_{mEVP}},\beps^{N_{mEVP}},\sig^{N_{mEVP}})$

    \State Advect $A_h^n$ and $H_h^n$ with DG upwind RK2 using $\bu_h^{n+1}$
    \State Apply limiters to $A_h^{n+1}$ and $H_h^{n+1}$
\EndFor
\end{algorithmic}
\end{algorithm}


\section{Assessment on the sea ice benchmark}
\subsection{The sea ice benchmark}
The benchmark introduced in \cite{Mehlmann2017, MehlmannLKF2021} consists in solving the sea ice dynamics equations on a $L \times L =  512 \text{ km} \times 512 \text{ km}$ square with homogeneous Dirichlet (no slip) conditions for the sea-ice velocity, using prescribed wind --- moving anticyclone --- and ocean --- steady circular current --- forcings, and initial conditions for the sea-ice velocity, height and concentration. The simulation is ran for times $t \in [0, T]$, $T=2$ days. The prescribed ocean current is given by the field 
\begin{equation}
\bu_o(x,y)=\frac{\bu_o^\text{max}}{L}
\begin{pmatrix}
2y-L \\
-2x+L
\end{pmatrix},
\label{eq:ocean_velocity}
\end{equation}with $\bu_o^\text{max} = 0.01$ m.s$^{-1}$. The moving anticyclone field is 
\begin{equation}
\bu_a(x,y,t)=
-\frac{30}{100 \text{ km}} \exp\left(
-\frac{|\mathbf{x} -\mathbf{m}(t)|}{100 \text{ km}}
\right)
\begin{pmatrix}
\cos (\theta) & \sin (\theta)\\
-\sin (\theta) & \cos (\theta)
\end{pmatrix} \left ( \mathbf{x}-\mathbf{m}(t)\right ),
\label{eq:wind_velocity}
\end{equation}
with $\mathbf{m}(x,y) = 256\,\unit{km} + 51.2\,\unit{km.day^{-1}}$ the center of the anticyclone,  $\theta=72^\circ$ and the maximal wind speed $\bu_a^\text{max} = \frac{30}{e} \approx 11\,\unit{m.s^{-1}}$. Initial conditions for the ice velocity and tracers are given by
\begin{align}
    \bu(t = 0, x, y) = \bu_0(x, y) &= 0\,\unit{m.s^{-1}},\\
     A(0, x, y) = A_0(x,y)& = 1,\\
     H(0, x, y) = H_0(x,y)& = 0.3\,\unit{m}+ 0.005\,\unit{m}\left ( \sin\left(\frac{60\, x}{1000\,\unit{km}}\right) + \sin \left (\frac{30\, y}{1000\,\unit{km}}\right ) \right ).
\end{align}
The parameters for the mEVP model are given in Table \ref{tab:mevp_parameters}. \begin{table}[h]
\centering
\caption{Physical and numerical parameters for the mEVP model, taken from \cite{MehlmannLKF2021}.}
\label{tab:mevp_parameters}
\begin{tabular}{lll}
\toprule
Parameter & Value & Description \\
\midrule
$\rho_{ice}$ & $\SI{900}{kg.m^{-3}}$ & Ice density \\
$\rho_a$ & $\SI{1.3}{kg.m^{-3}}$ & Air density \\
$\rho_o$ & $\SI{1026}{kg.m^{-3}}$ & Ocean density \\
$C_a$ & $1.2\times 10^{-3}$ & Air drag coefficient \\
$C_o$ & $5.5\times 10^{-3}$ & Ocean drag coefficient \\
$f_c$ & $\SI{1.46e-4}{s^{-1}}$ & Coriolis parameter \\
$P^\star$ & $\SI{27.5e3}{N.m^{-2}}$ & Ice strength parameter \\
$C$ & $20$ & Ice concentration parameter \\
$e$ & $2$ & Elliptic yield-curve eccentricity \\
$\Delta_{\min}$ & $\SI{2e-9}{s^{-1}}$ & Viscous regularization threshold \\
\bottomrule
\end{tabular}
\end{table}

\subsection{Numerical results}
We evaluate the LDG sea ice model by comparing the obtained shear deformation
$$\beps_{II} = \sqrt{(\beps_{11} - \beps_{22})^2 + 4 \beps_{12}^2},
$$with the corresponding results obtained using the neXtSIM-DG \cite{richter2023dynamical}, ICON~\cite{Mehlmann2021,iconorg}, and Gascoigne~\cite{Braack2021} CR-Q0 models. ICON and Gascoigne data were taken from the open access repository \url{https://data.mendeley.com/datasets/kj58y3sdtk/1} released with publication \cite{MehlmannLKF2021}. 

The LDG method was implemented with the python NGsolve software, and computed with $\Delta t=6$ min, $\alpha = \beta = \frac{1}{\Delta \tau} = 1000$ and $N_{mEVP} = 400$ subiterations. Results are presented for order 1 (\textbf{LDG-P1}, \textbf{LDG-Q1}) or order 2 (\textbf{LDG-P2}, \textbf{LDG-Q2}) velocity DG approximation spaces, and compatible stress and strain spaces (see subsection \ref{sec:compatibility}).

The neXtSIM-DG (\textbf{cG/DG}) method uses a continuous Q1 or Q2 finite element discretization for the velocity $\bu$ associated with a DG space for the stress tensor $\sig$, and a DG upwind scheme for the tracers on quadrilateral meshes. For this paper, order 0 was taken for the tracers. The physical time step is $\Delta t=2$ min, and the mEVP parameters are $\alpha = \beta = 1500$ for $N_{mEVP} = 100$ subiterations.

Both Gascoigne CR-Q0 and ICON solvers use nonconforming finite elements for the velocity; rotated bilinear element \cite{https://doi.org/10.1002/num.1690080202} on a quadrilateral mesh for Gascoigne CR-Q0, and Crouzeix-Raviart elements on a triangular mesh for ICON (see details in \cite{Mehlmann2021} and \cite{MehlmannLKF2021}). This type of discretization, which belongs to the Arakawa CD-grid class and allows  discontinuities in the velocity field, has already proven to resolve more LKFs than classical first order continuous finite elements or finite volumes discretization schemes for the momentum equation \cite{MehlmannLKF2021, Mehlmann2021}. Both Gascoigne CR-Q0 and ICON use a classical upwind scheme for the tracers. The data available from \cite{MehlmannLKF2021} are computed with a time step $\Delta t = 2$ min and $N_{mEVP} = 100$ for ICON. Gascoigne solves implicitely the viscous-plastic momentum equation \eqref{eq:momentum} \cite{Mehlmann2017}.

\begin{figure}[h!]
    \centering
        \includegraphics[width=\textwidth]{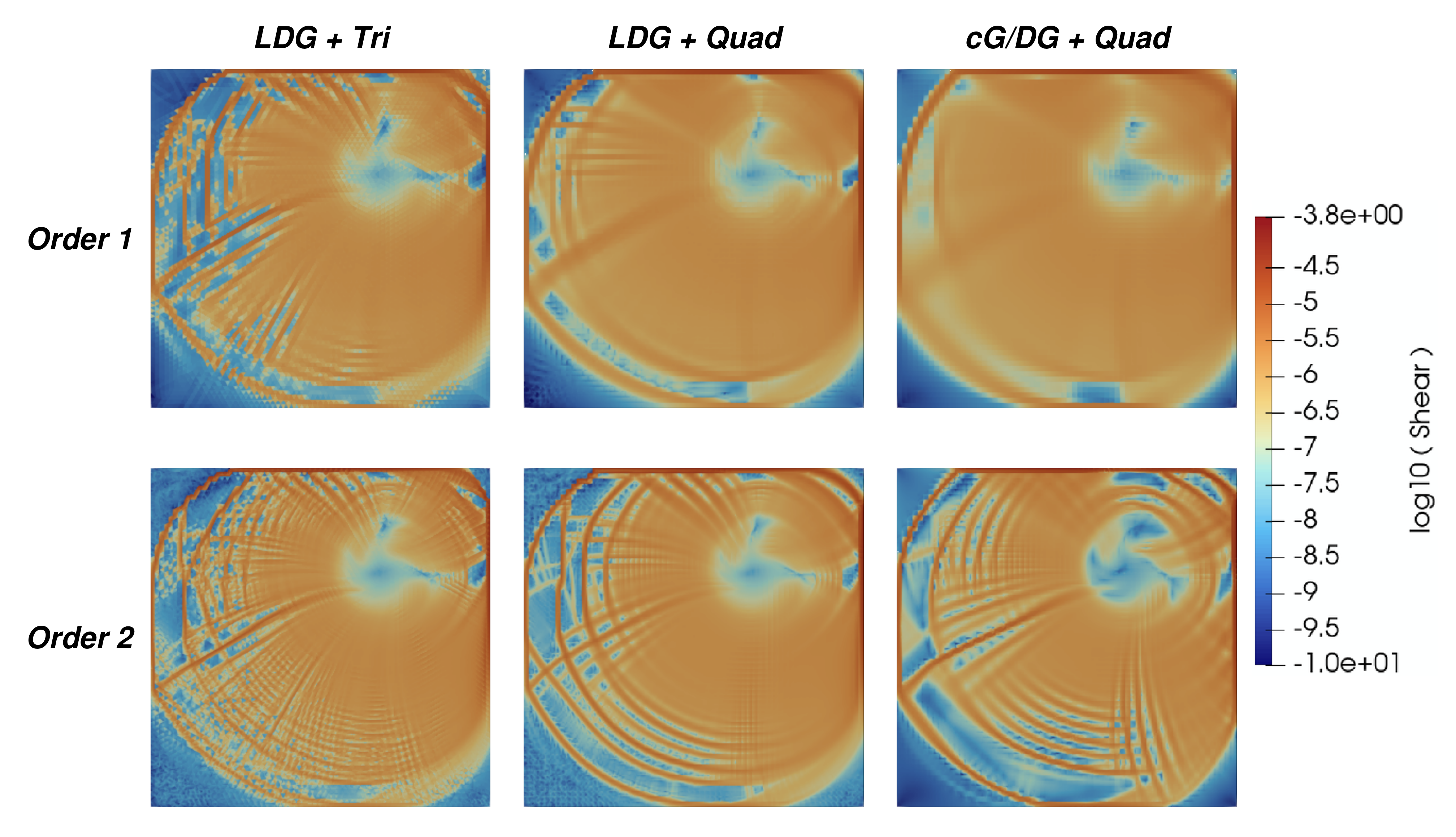}
        \caption{Simulated shear deformation at time $t=2$ days given in logarithmic scale for a $8\,\unit{km}$ grid, using the LDG model with approximations spaces of order 1 and 2 (top and bottom row) on triangular and quadrilateral elements, and  neXtSIM-DG (third column).}
        \label{8km}
\end{figure}
\begin{figure}
    \begin{subfigure}[b]{\textwidth}
        \includegraphics[width=1\textwidth]{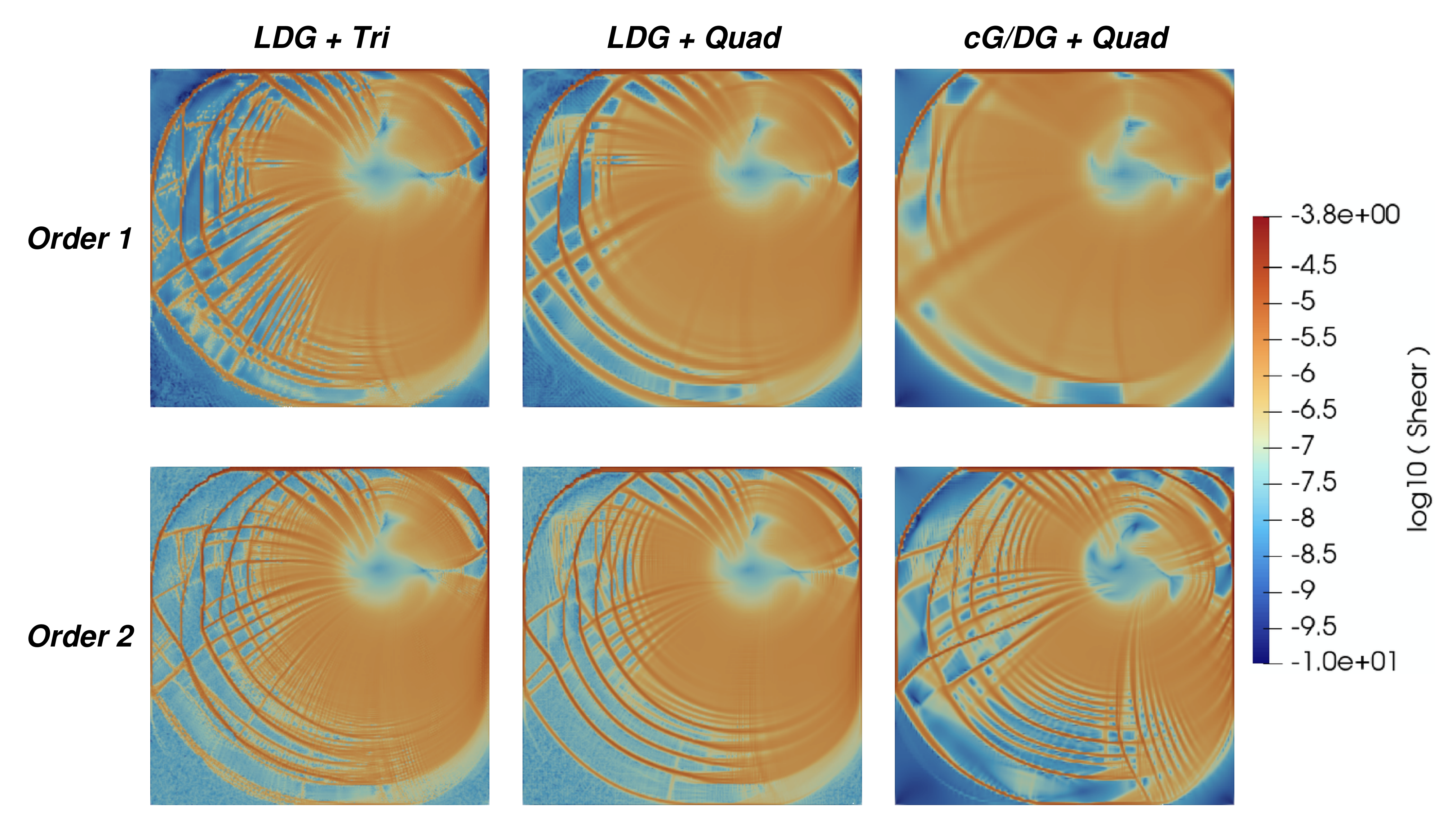}
        \caption{4 km mesh size}
        \label{4km}
    \end{subfigure}%
    \hfill
    \begin{subfigure}[b]{\textwidth}
        \centering
        \includegraphics[width=1\textwidth]{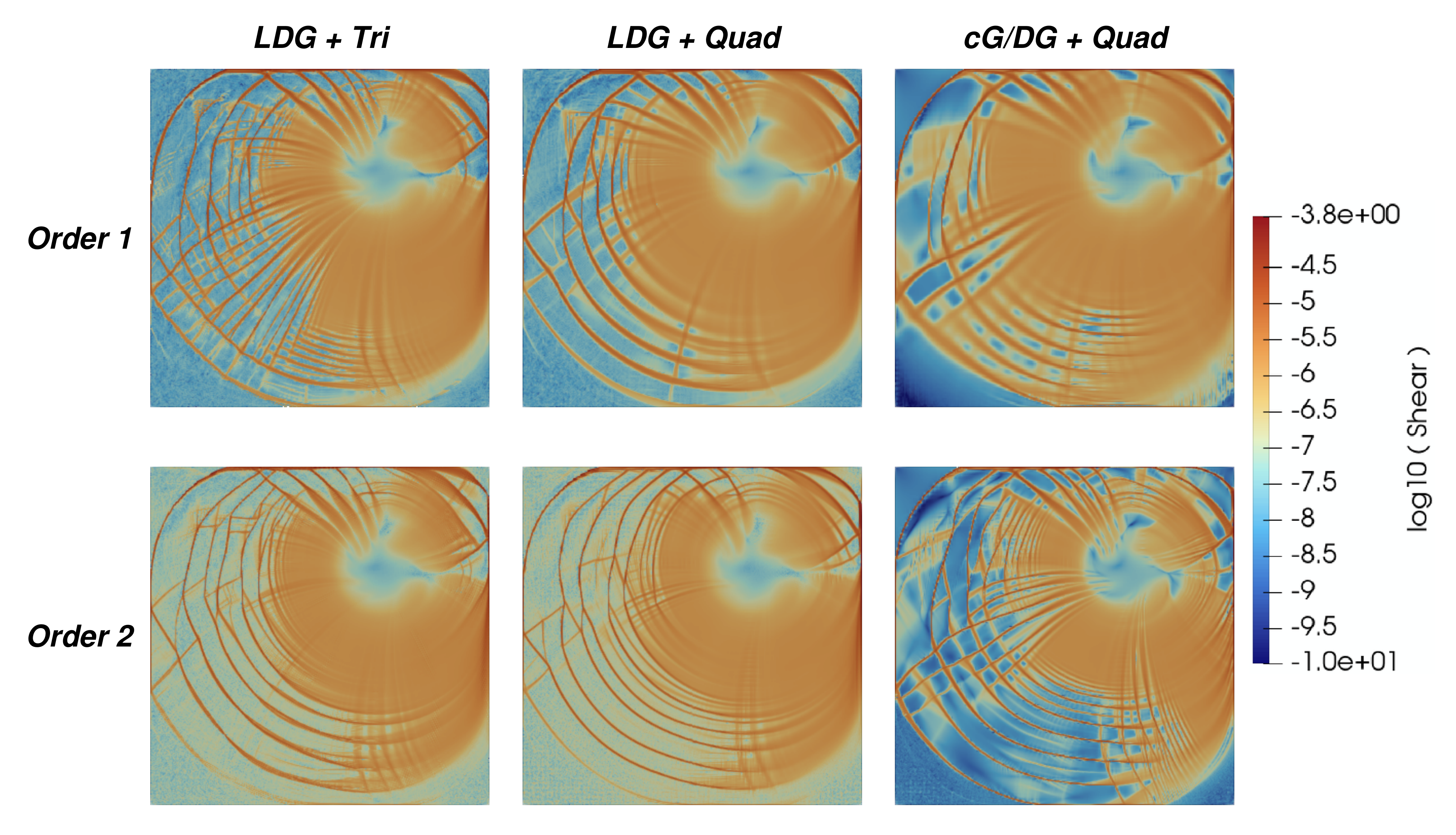}
        \caption{2 km mesh size}
        \label{2km}
    \end{subfigure}
    \caption{Simulated shear deformation at time $t=2$ days given in logarithmic scale for 4 km and 2 km grids, using the LDG model with approximations spaces of order 1 and 2 (top and bottom row) on triangular and quadrilateral elements, and the neXtSIM-DG software (third column), which uses continuous approximation spaces for the velocity and discontinuous spaces for the stress tensor on quadrilateral elements.}
\end{figure}
\begin{figure}[h!]
    \centering
        \includegraphics[width=\textwidth]{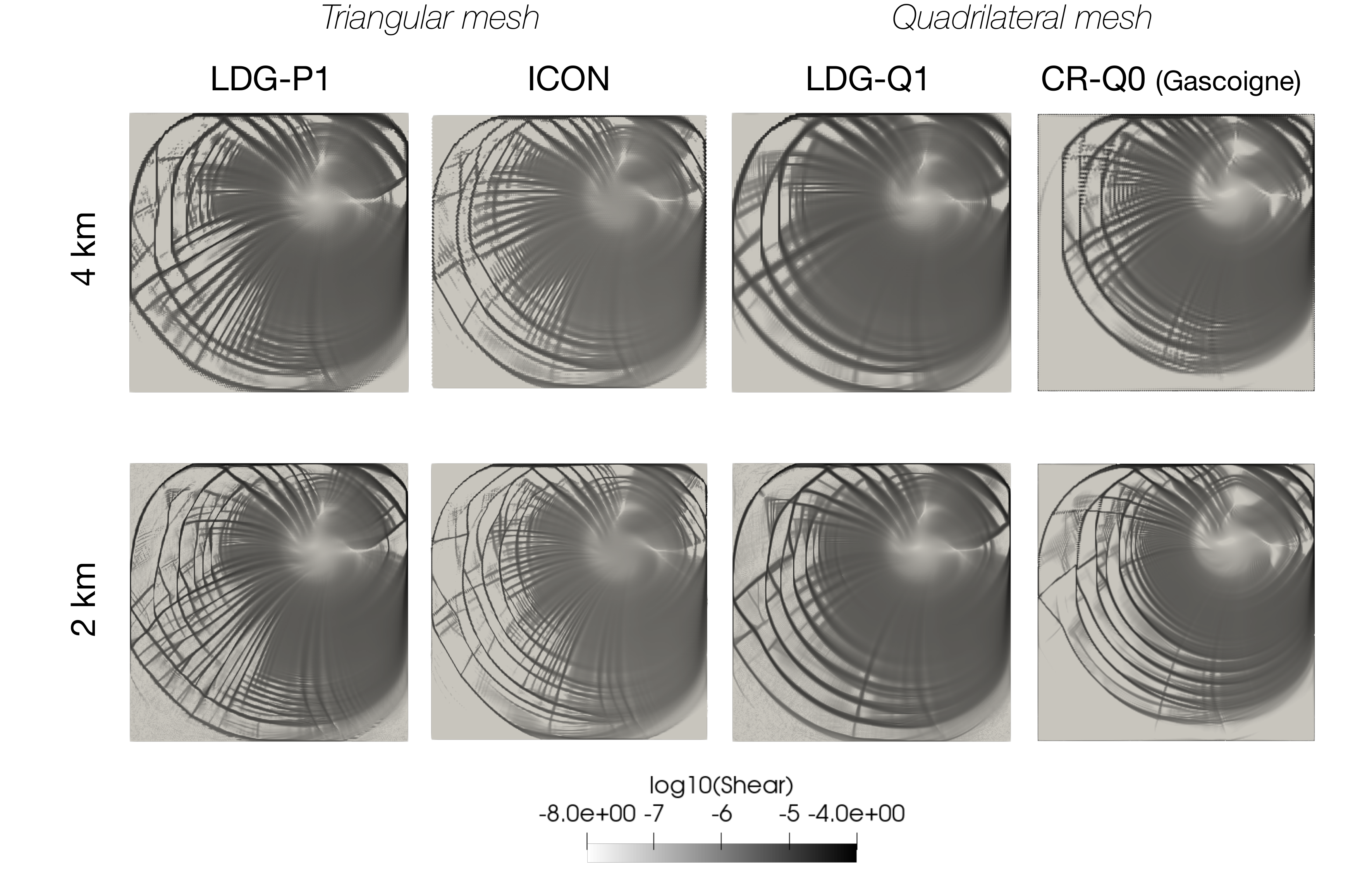}
        \caption{Simulated shear deformation at $t=2$ days obtained with the ICON and \textbf{LDG-P1} models on triangular 2 and 4 km grids, and with the Gascoigne CR-Q0 and \textbf{LDG-Q1} models on quadrilateral 2 and 4 km grids. ICON and Gascoigne data were taken from the open access repository \url{https://data.mendeley.com/datasets/kj58y3sdtk/1} released with publication \cite{MehlmannLKF2021}. }
        \label{fig:ICON LDG CR o1}
\end{figure}
\begin{figure}[h!]
    \centering
        \includegraphics[width=\textwidth]{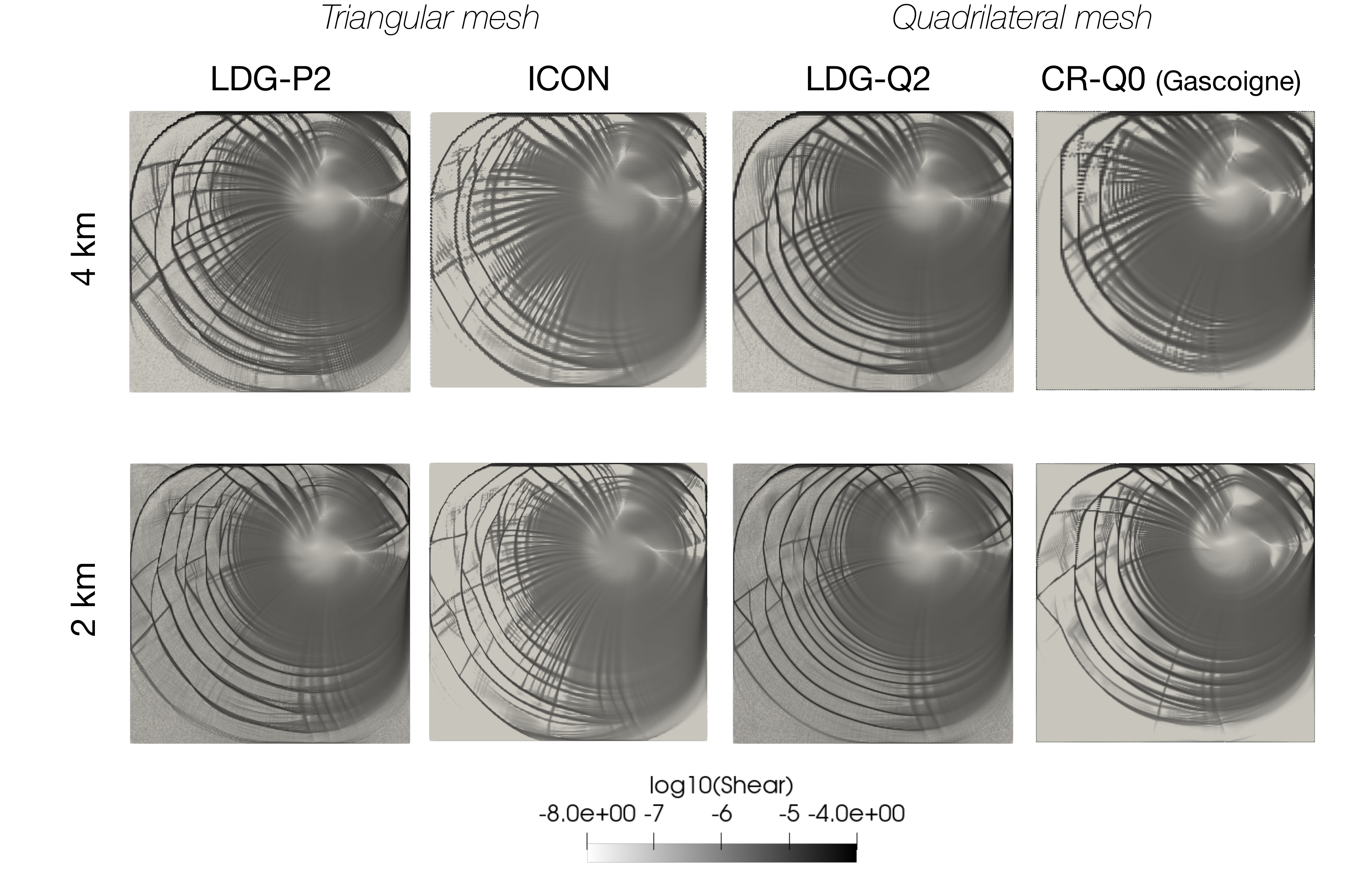}
        \caption{Simulated shear deformation at $t=2$ days obtained with the ICON and \textbf{LDG-P2} models on triangular 2 and 4 km grids, and with the Gascoigne CR-Q0 and \textbf{LDG-Q2} models on quadrilateral 2 and 4 km grids. ICON and Gascoigne data were taken from the open access repository \url{https://data.mendeley.com/datasets/kj58y3sdtk/1} released with publication \cite{MehlmannLKF2021}. }
        \label{fig:ICON LDG CR o2}
\end{figure}
Figures \ref{8km}, \ref{4km}, and \ref{2km} compare the first-order \textbf{LDG-P1} and \textbf{LDG-Q1} solutions, as well as the second-order \textbf{LDG-P2} and \textbf{LDG-Q2} solutions, with the corresponding neXtSIM-DG outputs on 8, 4 and 2 km grids, respectively. Figures \ref{fig:ICON LDG CR o1} and \ref{fig:ICON LDG CR o2} present comparisons between the LDG, ICON, and Gascoigne CR-Q0 results on 4 and 2 km grids. Figure \ref{fig:ICON LDG CR o1} shows the first-order LDG solutions, whereas Figure \ref{fig:ICON LDG CR o2} presents the second-order LDG solutions.

\paragraph{Remark} \textit{Solutions on quadrilateral meshes, represented elementwise by bilinear (Q1) or biquadratic (Q2) basis functions, present smoother shear and stress fields than triangular meshes solutions of the same order, which are linear (P1) or quadratic (P2) by element. 
However, owing to the compatibility requirements between the velocity and strain spaces discussed in subsection \ref{sec:compatibility}, quadrilateral discretizations require larger strain spaces, namely $(W^1(\mathcal{T}_h))$ or $(W^2(\mathcal{T}_h))$, while triangular discretizations use $(W^0(\mathcal{T}_h))$ or $(W^1(\mathcal{T}_h))$, resulting in an increased computational cost. }

\paragraph{Advantage of the discontinous representation of variables} 
The differences between the fully discontinuous LDG discretization and the hybrid continuous–discontinuous (neXtSIM-DG) approach are most pronounced on coarse meshes and for low-order approximations (see Figures \ref{8km}, \ref{4km} and \ref{2km}). In particular, the LDG solutions exhibit more and sharper linear kinematic features (LKFs), owing to the natural discontinuities permitted by the velocity space. Such discontinuities are already visible on coarse 8 km meshes at low order, especially for triangular elements, although some of the resulting LKFs appear to be numerical artifacts, as discussed below. As the mesh is refined or the approximation order is increased, the differences between the discretizations become less marked, consistently with the expectation that all approximation spaces converge toward the same limiting solution. For the high-order LDG solutions on the 2 km mesh (Figure \ref{2km}), weak background noise is observed in regions of nearly vanishing shear deformation. This noise may result from a pseudo-time step $\Delta\tau$ that is slightly too large, causing small instabilities in the mEVP inner iterations; however, it does not noticeably affect the resolved LKFs. Nevertheless, the LDG formulation produces sharper and more localized shear-deformation bands than those obtained with neXtSIM-DG, even with high order approximation and thinner mesh, as expected from its fully discontinuous representation.

Compared with the ICON and Gascoigne CR-Q0 models, the LDG method produces a similar number of LKFs at both first and second order, as shown in Figures \ref{fig:ICON LDG CR o1} and \ref{fig:ICON LDG CR o2}. ICON and Gascoigne CR-Q0 already generate numerous sharp LKFs as a consequence of the nonconforming elements employed, which also allowed (partial) discontinuities in the velocity space (see \cite{Mehlmann2021}).

\paragraph{Robustness with respect to the grid type}
The use of both triangular and quadrilateral meshes makes it possible to assess the sensitivity of the LDG and nonconforming Crouzeix--Raviart discretizations (and more generally of the numerical sea-ice dynamics) to the grid type and approximation order. Comparing the LDG solutions obtained on quadrilateral and triangular meshes in Figures \ref{8km}, \ref{4km}, and \ref{2km}, we observe that, at first order, the orientation and location of the LKFs differ substantially between the two grid types. These differences remain consistent across the three mesh resolutions. Moreover, the LDG solutions on quadrilateral meshes share more LKFs patterns with the neXtSIM-DG solutions, which are also computed on quadrilateral grids, than with the same order LDG solutions on triangular meshes.

A similar grid dependence is apparent in the comparison with ICON, based on triangular elements, and Gascoigne CR-Q0, based on quadrilateral elements. As shown in Figure \ref{fig:ICON LDG CR o1}, the first-order LDG solutions exhibit patterns resembling those produced by the corresponding nonconforming discretization on the same grid type. Consequently, the triangular and quadrilateral mesh solutions form two distinct groups in terms of LKF orientation and spatial distribution.

These observations also suggest that the total number or cumulative length of LKFs, quantities commonly used to compare numerical methods for sea ice dynamics, may not be reliable indicators when first-order approximations are considered. Indeed, numerous apparently nonphysical features are present in the first-order solutions. Here, we classify as numerical artifacts the LKFs that disappear when the approximation order and mesh resolution is increased. Such artifacts are particularly prominent on triangular meshes, notably in the ICON and \textbf{LDG-P1} solutions shown in Figure \ref{fig:ICON LDG CR o1}.

By contrast, the LKF patterns obtained with the second-order LDG discretizations, as well as with second order neXtSIM-DG, remain broadly consistent across grid types; see Figures \ref{fig:ICON LDG CR o2}, \ref{4km} and \ref{2km}. This suggests that mesh-aligned LKFs are primarily associated with first-order spatial approximations and that a second-order fully discontinuous Galerkin discretization can, accordingly to expectations, substantially reduce this grid-dependence in addition to improving the qualitative accuracy of the solutions, even on relatively coarse meshes.


\section{Discussion}
A first contribution of this work has concerned the convergence of the mEVP formulation. Previous studies established conditions for the stability of spatially discretized mEVP schemes and demonstrated numerically that the discrete mEVP solution converges toward the corresponding discrete VP solution, when a sufficiently large number of pseudo-time subiterations is performed. In this paper, we have complemented these results by studying the underlying pseudo-time continuous system independently of any particular spatial discretization. We proved its convergence toward the viscous-plastic limit at a fixed physical time. This result provides a theoretical justification for the use of mEVP as a pseudo implicit solver for the VP model.

In a second time, we introduced and assessed a fully discontinuous Galerkin formulation for solving Hibler's viscous-plastic sea ice model. The proposed method is based on the local discontinuous Galerkin framework and represents all variables, including the velocity, in discontinuous finite-element spaces. The method was evaluated using a standard sea ice dynamical benchmark and was shown to reproduce the expected sharp and numerous linear kinematic features, although at a higher computational cost.

The numerical experiments also highlighted the influence of the grid type and approximation order on the simulated ice deformation field. For first order polynomial approximation of the velocity, the number, orientation, and spatial distribution of the LKFs depend strongly on whether triangular or quadrilateral elements are used. Solutions obtained with different numerical methods but based on the same grid geometry exhibit similar LKF patterns, indicating that part of these structures may be oriented by the mesh rather than by the underlying physical model. Some features tend to disappear when the approximation order and the mesh resolution are increased and can therefore be interpreted as numerical artifacts rather than physically meaningful LKFs. These observations also question the validity of using the number and total length of LKFs to assess sea ice dynamics numerical modeling.

By contrast, the second-order LDG solutions exhibit greater consistency across grid types and spatial resolutions. Their LKF patterns remain qualitatively similar on triangular and quadrilateral meshes and are already well resolved on relatively coarse grids. These results suggest that mesh-aligned LKFs are primarily associated with low-order spatial approximations. They also demonstrate one of the main advantages of the DG framework: the local approximation order can be increased while preserving the flexibility of unstructured meshes. A second-order fully discontinuous discretization can therefore improve the qualitative robustness of the solution, and produce sharp, localized LKFs without requiring very fine meshes.

Future work will focus on optimizing the numerical implementation of the LDG scheme, as the use of discontinuous approximation spaces introduces additional degrees of freedom and interface terms compared with continuous formulations. A natural extension would consist in implementing the LDG method within the neXtSIM-DG software, thus naturally coupling with existing high order discontinuous Galerkin methods for the tracers advection equation.

\section*{Acknowledgments}
The authors gratefully acknowledge the funding by the European Regional Development Fund (ERDF) within the program Research and Innovation - Grant Number ZS/2023/12/182075 (Center for Dynamic Systems), as well as 
the fundings from the Deutsche Forschungsgemeinschaft (DFG, German Research Foundation), Grant number 314838170, GRK 2297 MathCoRe. TR  acknowledges support by Schmidt Sciences, Grant Number G-24-67790. 

\appendix
\section{Appendix}

\begin{lemma}[Bound for the Bregman divergence]
\label{lemma: bound Bregman}
Let $\psi(\bv)=\frac13 \|\bv\|_2^3.$
Then, for all \(\bv,\bv^*\in\mathbb R^2\) we have the following pointwise inequality
\[
    D_\psi(\bv,\bv^*)
    \geq
    \frac1{12}
    \left (\|\bv\|_2+\|\bv^*\|_2\right )\|\bv -\bv^*\|_2^2.
\]
\end{lemma}
\begin{proof}
Define the pointwise Bregman divergence associated with \(\psi\) by $ D_\psi(\bv,\bv^*)
    =
    \psi(\bv)-\psi(\bv^*)-\nabla\psi(\bv^*)\cdot(\bv-\bv^*).$
By Taylor's formula with integral remainder, and taking $h=\bv-\bv^*$, we have
\[
    D_\psi(\bv,\bv^*)
    =
    \int_0^1
    (1-t)\,
    h^T\nabla^2\psi(\bv^*+th)h\,dt.
\]
Since $\nabla^2\psi(\bv)=
    \|\bv\|_2 I+\frac{\bv\otimes \bv}{\|\bv\|_2} \succeq
    \|\bv\|_2 I$, we get
$$
    D_\psi(\bv,\bv^*)\geq \|h\|_2^2 \int_0^1
    (1-t)\|\bv^*+t h\|_2\,dt.
$$
Moreover, $\bv^*+t h=(1-t)\bv^*+t \bv$, then,
\[
\|(1-t)\bv^*+t\bv\|_2 \geq \bigl| (1-t) \|\bv^*\|_2-t\|\bv\|_2 \bigr|.
\]
Therefore,
\[
    D_\psi(\bv,\bv^*)
    \geq
    \|h\|_2^2 \int_0^1 (1-t)
    \bigl| (1-t)\|\bv^*\|_2-t\|\bv\|_2 \bigr|\,dt.
\]
We set $a=\|\bv^*\|_2\ge 0$ and $b=\| \bv\|_2 \ge0$.
If \(a=b=0\), then \(\bv=\bv^*=0\), and the result is immediate. Otherwise,
\[
    \bigl| (1-t)a-tb \bigr|
    =
    (a+b)\left| \frac{a}{a+b}-t\right|,
\]
hence we have
\[
    D_\psi(\bv,\bv^*)
    \geq
    (a+b)\|h\|_2^2
    \int_0^1
    (1-t)\left |\frac{a}{a+b}-t\right |\,dt.
\]
We can compute the integral to get
\[
    \int_0^1
    (1-t)\left |\frac{a}{a+b}-t\right |\,dt
    =
    -\frac{1}{3}\left (\frac{a}{a+b}\right )^3
    +\left (\frac{a}{a+b}\right )^2
    -\frac{1}{2} \left (\frac{a}{a+b}\right )
    +\frac16.
\]
The function $f: x \mapsto f(x) = -\frac{1}{3}x^3+x^2-\frac{1}{2} x+\frac16$ is bounded from below on \( [0,1]\). On this interval, its local extrema are localized in $0,\, 1$ and $1- \frac{\sqrt{2}}{2}$ with
\[f(x) = \int_0^1(1-t)|x-t|\,dt \geq \frac1{12},\quad \forall x\in [0, 1].
\]
Finally, we have the lower bound
\[
    D_\psi(\bv,\bv^*)
    \geq
    \frac1{12}
    \left (\|\bv\|_2+\|\bv^*\|_2\right )\|\bv -\bv^*\|_2^2.
\]
\end{proof}

\begin{prop}[Stability of the mEVP LDG scheme for a positive symmetric definite matrix rheology]\label{prop:stability}
For $\sig^{VP}(\beps) = L\beps$ with $L$ a symmetric definite positive, such that $(., L^{-1}.)$ defines a scalar product, the following energy estimate for the semi-discretized system \eqref{semi discrete mEVP} holds
    \begin{equation}
    \label{E bound}
    \begin{aligned}
E_h \leq E_h(0)e^{-2\tau} + \int_0^t e^{2\tau} \int_{\mathcal{T}_h} A\left ( \frac{2}{3}\|\bu_a\|^3 + \frac{1}{4}\|\bu_o\|^3 + \|\bu_a\|^2\|\bu_o\| \right )\,d\tau,
\end{aligned}
\end{equation}
for $E_h = \frac{1}{2}\|\bu_h\|^2 + \frac{1}{2}(\sig_h, L^{-1} \sig_h)$, $(., .)$ denoting the scalar product in $V^n(\mathcal{T}_h)$ or $W^m(\mathcal{T}_h)$.
\end{prop}
\begin{proof}
Choosing $\boldsymbol{\tau}_h = \boldsymbol{\sigma}_h$ in both \eqref{eps line sd}, $\mathbf{v}_h = \bu_h$ in \eqref{u line sd}, and $\bt_h = L^{-1} \boldsymbol{\sigma}_h$ in \eqref{sig line sd} yields the following
\begingroup
\begin{align}
 \sum_{K} ( \partial_\tau \mathbf{u}_h, \mathbf{u}_h)_K+ \sum_{K} (\mathbf{u}_h, \mathbf{u}_h)_K
+ \sum_{K} (\boldsymbol{\sigma}_h, \nabla \mathbf{u}_h)_K\qquad\qquad\nonumber \\ -\sum_{F \in \mathcal{F}_h^i} \langle \{\!\{\sig n\}\!\} - a \ljp \sig n \rjp
- b [\![\bu]\!]), \ljp \mathbf{u}_h  \rjp\rangle_{F}\qquad \nonumber \\§
-\sum_{F \in \mathcal{F}_h^b} \langle \boldsymbol{\sigma}_h \mathbf{n}- \frac{b}{0.5-a} \bu_h , \mathbf{u}_h \rangle_{F}
&= \sum_{K} (F(\bu_h), \mathbf{u}_h)_K,
\label{u line}\\
\sum_K (\partial_\tau \boldsymbol{\sigma}_h, L^{-1} \sig_h)_K
+\sum_K (\boldsymbol{\sigma}_h, L^{-1}\boldsymbol{\sigma}_h)_K
&= \sum_K (\beps_h, \sig_h)_K, \label{sig line}\\
\sum_{K} (\beps_h, \sig_h)_K - \sum_{K} (\nabla \mathbf{u}_h, \sig_h)_K \qquad \nonumber\\
 - \sum_{F \in \mathcal{F}_h^i}  \left[ \langle \{\!\{\bu \}\!\} + a \ljp \bu\rjp), \ljp \sig_h \mathbf{n} \rjp \rangle_F - \ljp \langle \mathbf{u}_h, \sig_h \mathbf{n} \rangle_F \rjp \right] + \sum_{F \in \mathcal{F}_h^b}   \langle \mathbf{u}_h, \sig_h \mathbf{n} \rangle_F& =0. \label{eps line}
\end{align}
\endgroup
Summing all three equations, we obtain the following estimate
\begin{equation}
\label{eq:final-energy}
\begin{aligned}
\frac{d}{d\tau}E_h + 2 E_h + b \sum_{F \in \mathcal{F}_h^i}
\|[\![\mathbf{u}_h]\!]\|^2_{L^2(F)}
+ \frac{b}{0.5-a} \sum_{F \in \mathcal{F}_h^b}
\|\mathbf{u}_h]\|^2_{L^2(F)}
=  ( F(\bu_h), \bu_h),
\end{aligned}
\end{equation} with $E_h = \frac{1}{2}\|\bu_h\|^2 + \frac{1}{2}(\sig_h, L^{-1} \sig_h)$, and $(., .)$ denoting the scalar product in $V^n(\mathcal{T}_h)$ or $W^m(\mathcal{T}_h)$. 

The term $F(\bu_h)$ writes (up to multiplicative constants with respect to $\bu_h$)
\begin{equation}
    F(\bu) = - \rho_\text{ice} H f\vec k \times \bu - A \|\bu - \bu_o\|_2 (\bu - \bu_o) + A \|\bu_a\|_2 \bu_a.
\end{equation}The first term is canceled by the scalar product with $\bu$ in the energy estimate: $(\vec k \times \bu) \cdot  \bu = 0$. Pointwise, we have
\begin{equation}
    \begin{aligned}
        F(\bu) \cdot \bu &= - A\|\bu - \bu_o\| (\bu-\bu_o)\cdot (\bu-\bu_o + \bu_o) + A \|\bu_a\| \bu_a \cdot \bu\\
        &\leq - A\|\bu - \bu_o\|^3 + A\|\bu - \bu_o\|^2 \|\bu_o\| + A \|\bu_a\|^2 \|\bu\|\\
        & \leq - A\|\bu - \bu_o\|^3 + A\|\bu - \bu_o\|^2 \|\bu_o\| +A \|\bu_a\|^2 \|\bu-\bu_o\| +A \|\bu_a\|^2\|\bu_o\|.
    \end{aligned}
\end{equation}Using generalized Young's inequality
$$a b \leq \frac{1}{(q\delta)^{p/q}}\frac{a^p}{p} + \delta\frac{b^q}{q} \quad \text{with }1/p + 1/q = 1 \text{ and }\delta>0 $$for 
$p=3/2, \quad q = 3,$ $a= \|\bu - \bu_o\|^2$, $b=\|\bu_o\|$ and $\delta = 3 \times 4$, we get
$$A\|\bu - \bu_o\|^2 \|\bu_o\| \leq A \left ( \frac19 \|\bu -\bu_o\|^3 + \frac14 \|\bu_o\|^3 \right ).
$$Now with $p=3, \quad q = 3/2,$ $a= \|\bu - \bu_o\|$, $b=\|\bu_a\|^2$ and $\delta = 1$, wet get for the air forcing terms
$$A \|\bu_a\|^2 \|\bu-\bu_o\| \leq A \left (\frac{2}{3} \|\bu_a\|^3 + \frac{1}{3}\|\bu-\bu_o\|^3 \right ).
$$Finally, we obtain the following bound
\begin{equation}
    \begin{aligned}
        F(\bu) \cdot \bu & \leq - \frac{5}{9}A\|\bu-\bu_o\|^3 + A \left ( \frac{2}{3}\|\bu_a\|^3 + \frac{1}{4}\|\bu_o\|^3 + \|\bu_a\|^2\|\bu_o\| \right ),
    \end{aligned}
\end{equation}which can be inserted in the energy estimate \eqref{eq:final-energy} to form the following differential inequality
\begin{equation}
    \label{ineq energy 1}
    \begin{aligned}
\frac{d}{d\tau}E_h + 2 E_h
\leq  \int_{\mathcal{T}_h} A\left ( \frac{2}{3}\|\bu_a\|^3 + \frac{1}{4}\|\bu_o\|^3 + \|\bu_a\|^2\|\bu_o\| \right ),
\end{aligned}
\end{equation}leading to the energy bound \eqref{E bound}.
\end{proof}


\printbibliography

\end{document}